\documentclass[11pt,oneside]{amsart}

\usepackage[a4paper]{geometry} 
\usepackage{amsmath, amssymb, amsthm, mathrsfs, amsthm, bbold, stmaryrd, euler, tikz-cd}
\usepackage{proof}
\usepackage{fouridx}
\usepackage{caption}
\usepackage{mathtools} 
\usepackage[all,cmtip,2cell]{xy} 
\UseTwocells
\usepackage{texdraw}
\everytexdraw{\drawdim pt \linewd 0.5 \textref h:C v:C \setunitscale 1}
\usepackage{url}
\usepackage[capitalise]{cleveref}

\theoremstyle{plain}
\newtheorem*{theorem*}{Theorem}
\newtheorem{theorem}{Theorem}[section]
\newtheorem{lemma}[theorem]{Lemma}
\newtheorem{proposition}[theorem]{Proposition}

\theoremstyle{definition}
\newtheorem{definition}[theorem]{Definition}
\newtheorem{construction}[theorem]{Construction}
\newtheorem{remark}[theorem]{Remark}
\newtheorem{example}[theorem]{Example}

\numberwithin{equation}{section}

\newcommand{\catfont}{\mathbb}
\newcommand{\catC}{\catfont{C}}

\newcommand{\catE}{\catfont{E}}
\newcommand{\catA}{\catfont{A}}

\newcommand{\Rmaps}{\text{\textbf{$R$-Map}}}
\newcommand{\Lmaps}{\text{\textbf{$L$-Map}}}
\newcommand{\Ralg}{\text{\textbf{$R$-Alg}}}
\newcommand{\Lalg}{\text{\textbf{$L$-Coalg}}}
\newcommand{\Ftmaps}{\text{\textbf{$F_t$-Map}}}
\newcommand{\Cmaps}{\text{\textbf{$C$-Map}}}
\newcommand{\Ftalg}{\text{\textbf{$F_t$-Alg}}}
\newcommand{\Calg}{\text{\textbf{$C$-Coalg}}}
\newcommand{\Fmaps}{\text{\textbf{$F$-Map}}}
\newcommand{\Ctmaps}{\text{\textbf{$C_t$-Map}}}
\newcommand{\Falg}{\text{\textbf{$F$-Alg}}}

\newcommand{\icat}{\mathbb}

\newcommand{\lcat}{\mathcal}
\newcommand{\ocat}{\mathscr}
\newcommand{\category}[1]{\underline{\smash[b]{\text{\rm{#1}}}}}
\newcommand{\scat}{\category}
\newcommand{\mono}{\rightarrowtail}
\newcommand{\htimes}{\hat{\times}}
\newcommand{\unufib}{u^{\bf n}_{\otimes}}
\newcommand{\nufib}{\lcat{I}^{\bf n}_{\htimes}}
\newcommand{\ufib}{\lcat{I}_{\htimes}}
\newcommand{\uufib}{u_{\otimes}}
\newcommand{\squish}{\textbf{squash}}
\newcommand{\squi}{\textbf{sq}}
\newcommand{\dom}{\text{dom}}
\newcommand{\colim}{\text{colim}}

\newcommand{\lorth}[1]{^{\boxslash}{#1}}
\newcommand{\co}{\colon}

\newcommand{\Jay}{\mathcal{J}}
\newcommand{\Ai}{\mathcal{I}}

\newcommand{\Em}{\mathcal{M}}

\newcommand{\CAT}{\text{\bf CAT}}
\newcommand{\GRP}{\scat{Gpd}}
\newcommand{\SSet}{\scat{sSet}}
\newcommand{\sSet}{\SSet}

\newcommand{\CSet}{\scat{cSet}}
\newcommand{\lift}{\text{lift}}
\newcommand{\abra}[1]{\langle #1 \rangle}

\newcommand{\hhom}{\hat{\hom}}
\newcommand{\commacat}[1]{{\bf \downarrow #1}}

\newcommand{\drpullback}[1][dr]{\save*!/#1-1.2pc/#1:(-1,1)@^{|-}\restore}
\newcommand{\drpushout}[1][ul]{\save*!/#1-1.2pc/#1:(-1,1)@^{|-}\restore}
\newcommand{\pbcorner}{\drpullback}
\newcommand{\pocorner}{\hbox to 8pt{{\vrule height8pt depth0pt width0.5pt}%
    \vbox to 8pt{{\hrule height0.5pt width7.5pt depth0pt}\vfill}}}
\newcommand{\poexcursion}{\save[]-<15pt,-15pt>*{\pocorner}\restore}

\begin{document}
\title[]{Models of Martin-L\"of type theory \\ from algebraic weak factorisation systems}
\date{\today}

\author[N. Gambino]{Nicola Gambino}
\address{School of Mathematics, University of Leeds, UK}
\email{n.gambino@leeds.ac.uk}

\author[M. F. Larrea]{Marco Federico Larrea}
\address{School of Mathematics Alumni, University of Leeds, UK}
\email{marco@lasi.ai}

\maketitle

\begin{abstract}
We introduce type-theoretic algebraic weak factorisation systems and show how they give rise to homotopy-theoretic models of Martin-L\"of type theory. This is done by showing that the comprehension
category associated to a type-theoretic algebraic weak factorisation system satisfies the assumptions necessary to apply a right adjoint method for splitting comprehension categories. 
We then provide methods for constructing several 
examples of type-theoretic algebraic weak factorisation systems, encompassing the existing groupoid
 and cubical sets models, as well as new models based on normal fibrations. 
\end{abstract}


\section*{Introduction}

{\bf Context and motivation} 
The construction of category-theoretic models of Martin-L\"of type theory~\cite{Nordstrom_2001} is a complex task that involves
two main problems. First, one needs to find a category with sufficient structure, so as to be able to interpret
the type-formation rules, e.g.~those for dependent sum and dependent product types
($\Sigma$-types and $\Pi$-types, respectively, for short). In particular, in order to have a model with
intensional identity types ($\mathsf{Id}$-types for short), the category under consideration needs to 
possess some homotopy-theoretic structure, as given for example by a weak factorisation system (wfs for short)~\cite{Awodey_2009,Gambino_2008,North_2019}. Secondly, one has to transform the category
under consideration into a genuine model of Martin-L\"of type theory, in which certain strictness conditions
(needed to model correctly the substitution operation) are required to hold, as in a split comprehension category~\cite{Jacobs_1993,Jacobs_1999}.
As these conditions are rarely satisfied in practice, this second step often involves applying suitable general coherence theorems,
analogous to Mac Lane's theorem relating monoidal categories and strict monoidal categories~\cite{MacLane_1963}. To make things more difficult, these two issues are closely related. 

There are two main methods to address the strictness issues, to which we refer as
the left and right adjoint splitting,~\cite{Curien,Hofmann_1994,Lumsdaine_2015}, since they are based on the left and the right adjoint to the inclusion of split Grothendieck fibrations into Grothendieck fibrations, respectively~\cite{giraud1966cohomologie,Streicher_2009}. 
The right adjoint splitting was  already used in~\cite{Hofmann_1994} in order to remedy the issues  affecting the interpretation of Martin-L\"of type theory in locally Cartesian closed categories \cite{Seely_1984}, thus accounting for $\Sigma$-types, $\Pi$-types and extensional $\mathsf{Id}$-types.
Subsequently, Warren isolated sufficient conditions to apply the right adjoint splitting
and produce models with intensional $\mathsf{Id}$-types~\cite{Warren_2016}. 
These conditions, however, are not generally satisfied in categories equipped with  weak factorisation systems, giving the impression, apparently widespread in the research community, that the right adjoint splitting cannot be used to construct homotopy-theoretic models of Martin-L\"of type theories and that  the left adjoint splitting should be used instead, applying the results in~\cite{Lumsdaine_2015}.

Our aim in this paper is to show that this impression is wrong and that the right adjoint splitting can be applied to obtain homotopy-theoretic models of Martin-L\"of type theory.
The key observation underpinning this work, which was suggested by Garner (see~\cite[p.~34]{Warren_2016}), is that the right adjoint splitting can be applied provided that we work with algebraic, rather than ordinary, weak factorisation systems (awfs's for short)~\cite{Garner_2009,Grandis_2006} (see also~\cite{Bourke_2016, Bourke_2016_II,Riehl}). In an awfs, the lifting properties that are part of the definition of a wfs are replaced by lifting structures, satisfying a naturality condition. As we will see, it is the algebraic character of awfs's that allows us to apply the right adjoint splitting. 

While awfs's may seem cumbersome,  there are situations in which it  is actually more natural to consider awfs's than wfs's, most notably in the ongoing work aimed at defining homotopy-theoretic models of Martin-L\"of's in a constructive metatheory~\cite{Cohen_2016,Gambino_2019,Gambino_2017}. Indeed, the left adjoint splitting used in~\cite{Lumsdaine_2015} seems to work only under certain exponentiability assumptions, which are not constructively valid in simplicial sets~\cite{Bezem_2015}. Issues of
constructivity are also the main motivation for our work on normal uniform fibrations, as explained further below.

\smallskip

\noindent{\bf Main results.} This paper makes two main contributions. The first is to introduce type-theoretic awfs's and show that they give rise to models of Martin-L\"of's type theory with $\Sigma$-types, $\Pi$-types and 
$\mathsf{Id}$-types. The second is to introduce a general method to obtain examples of type-theoretic awfs's and to give a homogeneous account of several models in which dependent types are interpreted as uniform fibrations (in the sense of~\cite{Cohen_2016,Gambino_2017}), including the groupoid model~\cite{Hofmann_1998}, in which dependent types are interpreted as split fibrations, and models based on simplicial and cubical sets~\cite{Cohen_2016,Gambino_2017}, in which dependent types are interpreted as uniform Kan fibrations. 

Our construction of models of Martin-L\"of type theory from type-theoretic awfs's is obtained in two steps. The first  is to define a non-split comprehension category from a type-theoretic awfs (\cref{prop:AWFSCC}) and show that this comprehension category is equipped with \emph{pseudo-stable} $\Sigma$-types, $\Pi$-types and $\mathsf{Id}$-types (\cref{thm:AMtoCC}) in the sense of~\cite{Lumsdaine_2015}, \emph{i.e.}~commuting with substitution up to isomorphism. The second step is to apply the right adjoint splitting and turn the non-split comprehension category obtained in the first step into a split one equipped with 
\emph{strictly stable} $\Sigma$-types, $\Pi$-types and $\mathsf{Id}$-types (\cref{thm:coht}). It should be noted that the extra algebraic structure of a type-theoretic awfs is crucial to have pseudo-stable $\mathsf{Id}$-types in the non-split comprehension category (cf.~\cite{Berg:2010aa}) and this, in turn, is essential to apply the  right adjoint splitting. As an illustration, we revisit the groupoid model of Martin-L\"of type theory \cite{Hofmann_1998}. Specifically, we equip the category of groupoids {\GRP } with an awfs whose right maps correspond to normal isofibrations and then we prove that such awfs is type-theoretical. In this way, we show that the original groupoid model can be obtained as the split comprehension category associated to this type-theoretic awfs. Since the appearance of this paper as a preprint, these results have been extended to
internal groupoids~\cite{Emmenegger_2020} and to split isofibrations~\cite{Woerkom}.

Our results on constructing type-theoretic awfs's build on the theory of uniform fibrations in \cite{Gambino_2017}. Our \cref{theorem:main} isolates  structure on Grothendieck topos that suffices to produce a type-theoretic awfs of uniform fibrations. We call a Grothendieck topos equipped with such structure a type-theoretic suitable topos (\cref{defn:ttst}). We then show that any Grothendieck topos equipped with an interval object with connections is an example of a type-theoretic suitable topos. The main technical 
machinery used in the proof of \cref{theorem:main} is \cref{theorem:main1}, where we show that given a type-theoretic suitable topos, the resulting awfs of uniform fibrations can be equipped with a stable functorial choice of path objects, which is the structure necessary to produce pseudo-stable identity types. This result fills
the gap between the theory developed in~\cite{Gambino_2017} and its intended application to the  construction of models of Martin-L\"of type theory.

We also advance the theory of \cite{Gambino_2017} by introducing a stronger version uniform fibrations, which we call normal uniform fibrations, in 
which the lifts are required to preserve degeneracies. These can be seen as generalisations of normal cloven isofibrations in groupoids, as explained in \cref{rem:nufiso}. We then show that the ideas
in~\cite{Gambino_2017} can be adapted so as to accommodate this new normality property. With this, we prove  that any suitable topos admits an awfs of normal uniform fibrations (\cref{thm:NUF}).

One of the reasons for our interest in normal uniform fibrations is that they allow us to avoid 
one of the hypotheses for a type-theoretic suitable topos~$\catE$ when constructing a type-theoretic
awfs (\cref{thm:mainNUF}). The requirement on~$\catE$ is that for any object $X \in \catE$, the reflexivity map~$r_X : X \rightarrow X^{I}$, that maps a point of $X$ to the constant path on it, is a member of a distinguished class of monomorphisms $\Em$ whose members are to be thought as generating cofibrations. While this assumption holds if we consider~$\Em$ to be the class of all monomorphisms, it fails in some situations that are of interest for constructive mathematics. For example, if $\catE$ is a presheaf topos and we restrict our attention to $\Em_{dec}$ the class of decidable monomophisms (i.e.~those whose image is level-wise constructively decidable). As noted in \cite{Orton_2017}, it is important to consider decidable monomorphisms when trying to model univalent universes. This issue is also relevant to the question of whether the path types and the identity types coincide in the cubical type theory of \cite{Cohen_2016}. 


\smallskip

\noindent
{\bf Outline of the paper.} \cref{sect:2} reviews the interpretation of type dependency using comprehension categories and the right adjoint splitting.
\cref{sect:2-bis} reviews algebraic weak factorisation systems.
 In \cref{sect:3} we introduce type-theoretic awfs's and prove that the induced comprehension categories support pseudo-stable $\Pi$-types, $\Sigma$-types and $\mathsf{Id}$-types. We then move on to the construction of examples of type-theoretic awfs. In \cref{sect:4} we revisit the groupoid model. In \cref{sect:UF} we show how to construct a type-theoretic awfs from a type-theoretic suitable topos.
 In \cref{sect:NUF} we introduce the awfs of normal uniform fibrations and in \cref{sect:TTNUF} we show that it is type-theoretic.
 
 \smallskip

\noindent 
{\bf Acknowledgements.} Nicola Gambino gratefully acknowledges the support of EPSRC under grant EP/M01729X/1, of the US Air Force Office for Scientific Research under agreement FA8655-13-1-3038 and the hospitality of the School of Mathematics of the University of Manchester while on study leave from the University of Leeds. Marco Federico Larrea would like to thank the Mexican National Council for Science and Technology (CONACyT) for the financial support during his PhD. We are grateful to Steve Awodey, John Bourke, Thierry Coquand, Peter Lumsdaine, Paige North (who spotted an inaccuracy in an earlier
version of~\cref{prop:sigmaAWFS}), Andy Pitts, Christian Sattler, Raffael Stenzel, Andrew Swan and Thomas Streicher for helpful discussions. We are also thankful to the referee for helpful suggestions.

\section{Strict stability, pseudostability and coherence} 
\label{sect:2}

We review basic notions and results on Grothendieck fibrations and comprehension categories,
referring to~\cite{Jacobs_1999,Streicher_2009} for more information. A \emph{comprehension category} over a category $\catC$ consists of a strictly commutative diagram  of the form
  \[
  \xymatrix {
  \catE \ar[rr]^{\chi} \ar[dr]_{\rho} & & \catC^{\rightarrow} \ar[dl]^{\text{cod}} \\
   & \catC \rlap{,} &
  }
  \]
$\text{cod} \co \catC^{\rightarrow}\rightarrow \catC$ is the codomain functor, such that $\rho \co \catE \rightarrow \catC$ is a Grothendieck fibration, and $\chi \co \catE \rightarrow \catC^{\rightarrow}$ maps Cartesian arrows in $\catE$ to pullback squares in~$\catC$. We refer to such a comprehension category by the tuple $(\catC, \rho, \chi)$ or by the pair $(\rho, \chi)$ if the category~$\catC$ is easily inferable from the context. A \emph{cleavage} for $(\catC, \rho, \chi)$ consists of a choice of lifts for the fibration $\rho$, \emph{i.e.} for each $u \co \Delta \rightarrow \Gamma$ in $\catC$ and $A$ over $\Gamma$, a Cartesian 
morphism~$u^* \co~A[u]~\rightarrow~A$ over $u$. We will refer to $A[u]$ as the \emph{reindexing} of $A$ along $u$. A cleavage is \emph{normal} if it preserves identities and is \emph{split} if it preserves identities and composition. A \emph{split comprehension category} is a comprehension category equipped with a split cleavage. Occasionally, we will make use the following notation:
\[
\xymatrix{
B \ar[r]^f \ar@{.}[d] & A \ar@{.}[d] && B \ar[r]^f \ar@{.}[d] \drpullback & A \ar@{.}[d] \\
\Delta \ar[r]_{\sigma} & \Gamma && \Delta \ar[r]_{\sigma} & \Gamma \rlap{,} 
} 
\]
where the diagram on the left indicates that $f$ is an arrow in $\catE$, $\sigma$ is an arrow in~$\catC$ 
and~$\rho(f) = \sigma$. The pullback notation on the diagram on the right indicates that, in addition to the previous data, $f$ is Cartesian.

A split comprehension category $(\catC, \rho, \chi)$ provides a natural setting to interpret the basic judgements and the structural rules of a dependent type theory~\cite{Jacobs_1999}.  Type-theoretic
contexts are interpreted as  objects of $\catC$. 
A judgement $\Gamma \vdash A \co \mathsf{type}$ is interpreted as an object~$A$ in the fibre of~$\rho$ over~$\Gamma \in \catC$. Context extension is modelled via the comprehension functor: for an object~$A$ in the fibre over~$\Gamma$ there is a morphism~$\chi_A \co \Gamma.A \rightarrow \Gamma$ whose domain $\Gamma.A$ is the interpretation of the context extension. A judgement $\Gamma \vdash a \co A$ is interpreted as a map $a \co \Gamma \to \Gamma.A$ in $\catC$ which is a section of $\chi_A \co \Gamma.A \to \Gamma$.

Substitution of terms into types is interpreted with the use of the split cleavage, while substitution of terms into terms is interpreted by composition. Weakening is modelled by reindexing an object $A$ along a morphism of the form $\chi_B : \Gamma.B \rightarrow \Gamma$. For simplicity, the comprehension of $A[\chi_B]$ is written  $\chi_{B, A} \co \Gamma.B.A \rightarrow \Gamma.B$. 

In order to model type-theoretic constructs, we  require a split comprehension category to be equipped with  additional, chosen, structure. Furthermore, this structure must cohere \emph{strictly} with the split cleavage, so as  to ensure validity of substitution rules. \cref{defn:idp} and \cref{defn:strictlystable} describe 
the structure necessary to model $\mathsf{Id}$-types. The corresponding structure for~$\Sigma$-types and~$\Pi$-types is defined in Appendix~\ref{sect:appendix}.


\begin{definition} \label{defn:idp}
A \emph{choice of $\mathsf{Id}$-types} on a comprehension category $(\catC, \rho, \chi)$ consists of the following data.
\begin{enumerate}
\item For each $A$ in the fibre over $\Gamma$ and $a, b \co \Gamma \to \Gamma.A$ sections of $\chi_A$, an object~$\mathsf{Id}_A(a, b)$ over $\Gamma$.
\item For each $A$ over $\Gamma$, a section $r_A$ over the diagonal morphism $\delta_A$, giving a factorisation:
\[
\xymatrix{
&& \Gamma.A.A.\mathsf{Id}_A(x, y) \ar[d]^{\chi_{\mathsf{Id}_A(x, y)}} \\
\Gamma.A \ar[rr]_-{\delta_A} \ar[urr]^{r_A} && \Gamma.A.A \rlap{,}
}
\]
where $\mathsf{Id}_A(x, y)$ is given by $(1)$ applied to the weakened type $A$ over $\Gamma.A.A$ and to the canonical variables $x, y \co \Gamma.A.A \to \Gamma.A.A.A$.
\item For any $A$ over $\Gamma$, $C$ over $\Gamma.A.A.\mathsf{Id}_A(x, y)$ and $t$ a section of $C$ over $r_A$ as shown by the solid arrows in the following diagram:
\[
\xymatrix{
\Gamma.A \ar[d]_{r_A} \ar[rr]^{t} && \Gamma.A.A.\mathsf{Id}_A(x, y).C \ar[d]^{\chi_C} \\
\Gamma.A.A.\mathsf{Id}_A(x, y) \ar@{.>}[urr]|{j_A(C, t)} \ar@{=}[rr] && \Gamma.A.A.\mathsf{Id}_A(x, y)  \rlap{,}
}
\] 
a section $j_A(C, t)$ of $C$ (shown as the dotted arrow) making both triangles commute. 
\end{enumerate}
\end{definition}

We refer to a choice of $\mathsf{Id}$-types by $(\mathsf{Id}, r, j)$. Similarly, we refer to a 
choice of $\Sigma$-types as $(\Sigma, pair, sp)$ and to a choice of $\Pi$-types as $(\Pi, \lambda, 
app)$ (see Appendix~\ref{sect:appendix} for details).

\begin{definition} \label{defn:strictlystable}
Assume $(\catC, \rho, \chi)$ is split and is equipped with a choice $(\mathsf{Id}, r, j)$ of $\mathsf{Id}$-types. We say that the choice is \emph{strictly stable} if for every morphism $\sigma \co \Delta \rightarrow \Gamma$ in the base category, and for every object $A$ in the fibre over $\Gamma$, the following conditions are satisfied:
\begin{enumerate}
\item For any pair of sections $a, b \co \Gamma \to \Gamma.A$, we have that $\mathsf{Id}_{A[\sigma]}(a[\sigma], b[\sigma]) = \mathsf{Id}_A(a, b) [\sigma]$.
\item The following diagram commutes
\[
\xymatrix{
\Delta.A[\sigma] \ar[rr]^{\sigma^*} \ar[d]_{r_{A[\sigma]}} && \Gamma.A \ar[d]^{r_A} \\
\Delta.A[\sigma].A[\sigma].\mathsf{Id}_{A[\sigma]}(x', y') \ar[rr]_{\sigma^{***}} && \Gamma.A.A.\mathsf{Id}_A(x, y) \rlap{,}
}
\]
where the horizontal arrows are given by split reindexing along $\sigma$
\item For any $(C, t)$ as in $(3)$ of \cref{defn:idp}, we can obtain a second pair $(C[\sigma], t[\sigma])$ relative to $\mathsf{Id}_{A[\sigma]}(x', y')$ by reindexing along $\sigma$. We require following diagram to commute:
\[
\xymatrix{
\Delta.A[\sigma].A[\sigma].\mathsf{Id}_{A[\sigma]}(x', y') \ar[d]_{j_{A[\sigma]}(C[\sigma], t[\sigma])} \ar[rrr]^{\sigma^{***}} &&& \Gamma.A.A.\mathsf{Id}_A(x, y) \ar[d]^{j_A(C, t)} \\
\Delta.A[\sigma].A[\sigma].\mathsf{Id}_{A[\sigma]}(x'. y').C[\sigma] \ar[rrr]_{{\sigma^{****}}} &&& \Gamma.A.A.\mathsf{Id}_A(x, y).C \rlap{,}
}
\]
where the horizontal arrows are given by reindexing along $\sigma$.
\end{enumerate}
\end{definition}

While split comprehension categories provide a sound interpretation of type theory, non-split comprehension categories arise most naturally in examples. This mismatch can be remedied 
by applying a well-known construction by Giraud and B\'{e}nabou \cite{giraud1966cohomologie} that replaces a comprehension category $(\catC, \rho, \chi)$ with an equivalent split one $(\catC, \rho^R, \chi^R)$, universally as a right adjoint functor. 
We review this construction. 

\begin{definition} \label{localcleavage}
Let $(\catC, \rho, \chi)$ be a comprehension category and $A$ an object in the fibre over $\Gamma \in \catC$. A \emph{local cleavage} for $A$ consists of an operation $A[-]$ that assigns to each map~$\sigma : \Delta \rightarrow \Gamma$ a Cartesian arrow $\sigma^* \co A[\sigma] \rightarrow A$ over $\sigma$. We say that a local cleavage~$A[-]$ is \emph{normal} if when applied to the identity $1_{\Gamma} \co \Gamma \rightarrow \Gamma$ it outputs the identity arrow, {i.e.}~$A[1_{\Gamma}] = A$ and~$1_\Gamma^* = 1_A$.
\end{definition}

For a fibration $\rho \co \catE \rightarrow \catC$  the category $\catE^R$ is defined as follows. Its objects are pairs~$(A , A[-])$, where~$A$ is an object of~$\catE$ and~$A[-]$ is a local normal cleavage for~$A$. An arrow $f \co (B , B[-]) \rightarrow (A , A[-])$ is just an arrow $f \co B \rightarrow A$ in $\catE$.  
Composition and  identities are just those of $\catE$. Notice that there is a functor $\rho^R \co \catE^R \rightarrow \catC$ given on objects by $\rho^R (A , A[-]) = \rho(A)$. The next lemma is 
well-known~\cite{Curien,giraud1966cohomologie}.

\begin{lemma} \label{lem:splitfibrep}
The functor $\rho^R \co \catE^R \rightarrow \catC$ is a split Grothendieck fibration.  \qed
\end{lemma} 


This construction extends to comprehension categories. Fix a comprehension category~$(\catC, \rho, \chi)$. First, we have a morphism $\epsilon_\rho \co \rho^R \rightarrow \rho$ of fibrations (i.e. a functor over~$\catC$ that preserves Cartesian arrows) that acts on an object $(A , A[-])$ of~$\catE^R$ by forgetting the local normal cleavage. Then, we obtain a split comprehension category $(\catC, \rho^R, \chi^R)$, by letting 
$\chi^R$ be the composite in
\[
\xymatrix{
\catE^R \ar[dr]_{\rho^R} \ar[r]^{\epsilon_{\rho}} & \catE \ar[d]^-{\rho} \ar[r]^{\chi} & \catC^{\rightarrow} \ar[dl]^{cod}\\
& \catC  \rlap{.} &  
}
\]

%

It is natural to ask what structure on a cloven comprehension category gives rise to strictly
stable $\Sigma$, $\Pi$, and identity types on its right adjoint splitting. As we discuss below, 
one needs a choice of $\Sigma$, $\Pi$, and $\mathsf{Id}$-types  which are \emph{pseudo-stable}.

\begin{definition} \label{defn:pseudostable}
A choice $(\mathsf{Id}, r, j)$ of $\mathsf{Id}$-types in a comprehension category is said to be \emph{pseudo-stable} if for any Cartesian arrow $f \co B \rightarrow A$ over a morphism $\sigma \co \Delta \rightarrow \Gamma$ in the base, the following conditions are satisfied:
\begin{enumerate}
\item For any pair of sections $a, b \co \Gamma \to \Gamma.A$, there is a Cartesian arrow 
\[
\mathsf{Id}_f(a, b) \co \mathsf{Id}_{B}(a[\sigma], b[\sigma]) \rightarrow \mathsf{Id}_A(a, b)\rlap{,} 
\] 
over $\sigma \co \Delta \rightarrow \Gamma$. Moreover, the assignment $f \mapsto \mathsf{Id}_f(a, b)$ is functorial, that is, $\mathsf{Id}_{1_A}(a, b) = 1_{\mathsf{Id}_A(a, b)}$ and $\mathsf{Id}_{f \circ g}(a, b) = \mathsf{Id}_{f}(a, b) \circ \mathsf{Id}_{g}(a[\sigma], b[\sigma])$.
\item The following diagram commutes:
\[
\xymatrix{
\Delta.B \ar[rr]^{f} \ar[d]_{r_{B}} && \Gamma.A \ar[d]^{r_A} \\
\Delta.B.B.\mathsf{Id}_{B}(x', y') \ar[rr]_{\mathsf{Id}_f(x, y)} && \Gamma.A.A.\mathsf{Id}_A(x, y)
}
\]
\item For any pair $(C, t)$ as in $(3)$ of \cref{defn:idp} and for any Cartesian $h \co C' \rightarrow C$ over $\mathsf{Id}_f$, we can construct a pair $(C', t')$ by pulling back $t$ along $h$ appropriately. We require that the following diagram commutes:
\[
\xymatrix{
\Delta.B.B.\mathsf{Id}_{B}(x', y') \ar[d]_{j_{B}(C', t')} \ar[rrr]^{\mathsf{Id}_f(x, y)} &&& \Gamma.A.A.\mathsf{Id}_A(x, y)\ar[d]^{j_A(C,t)} \\
\Delta.B.B.\mathsf{Id}_{B}(x', y').C' \ar[rrr]_{h} &&& \Gamma.A.A.\mathsf{Id}_A(x, y).C \rlap{,}
}
\]
where the lower horizontal arrow is the (comprehension of the) Cartesian arrow $h \co C' \rightarrow C$.
\end{enumerate}
\end{definition}

The definition of pseudo-stability for $\Sigma$-types and $\Pi$-types is given in Appendix~\ref{sect:appendix}.
The following coherence result connects pseudo-stability in a comprehension category and strict stability on its right adjoint splitting. 
The  proof is based on \cite[Theorem~2]{Hofmann_1994} for $\Sigma$-types and $\Pi$-types and on \cite[Theorem~2.48]{Warren_2016} for $\mathsf{Id}$-types. 

\begin{theorem} [Coherence Theorem] \label{thm:coht}
Let $(\catC, \rho, \chi)$ be a comprehension category equipped with pseudo-stable choices of $\Sigma$-, 
$\Pi$- and $\mathsf{Id}$-types. Then the right adjoint splitting $(\catC, \rho^R, \chi^R)$ is equipped with strictly stable choices of $\Sigma$-, $\Pi$- and $\mathsf{Id}$-types; and the counit $\epsilon_{\rho} \co (\catC, \rho^R, \chi^R) \rightarrow (\catC, \rho, \chi)$ preserves each choice of logical structure strictly. 
\end{theorem}

\begin{proof} 
%
For $\Sigma$-types, let us consider a dependent tuple (see  \cref{sect:appendix} for definition) $(\Gamma, (A, A[-]), (B, B[-]))$ of $(\rho^R, \chi^R)$. The $\Sigma$-type associated to this tuple has the following form:
$(\Sigma_A B, \Sigma_A B [-])$
where $\Sigma_A B$ is the $\Sigma$-type given by the pseudo-stable choice of $(\rho, \chi)$ applied to $(\Gamma, A, B)$. The component at $\sigma \co \Delta \rightarrow \Gamma$ of the local cleavage $\Sigma_A B [-]$ is given as follows. First, we use the local cleavages $A[-]$ and $B[-]$ to construct a Cartesian arrow of dependent tuples 
$(\sigma, f^*, g^*) \co (\Delta, A[\sigma], B[\sigma]) \rightarrow (\Gamma, A, B)$ where the arrows $f^*$ and $g^*$ are given by the local cleavage (see \cref{localcleavage}). Then, we use the action on morphisms of the pseudo-stable choice of $\Sigma$-types to define:
\[
\xymatrix{
\Sigma_A B [\sigma] \co= \Sigma_{A[\sigma]} B[\sigma] \ar[rr]^(.65){\sigma^* \co= \Sigma_{f^*} g^*} \ar@{.}[d] \drpullback && \Sigma_A B \ar@{.}[d] \\
\Delta \ar[rr]_{\sigma} && \Gamma
}
\] 
This local cleavage is normal because the pseudo-stable choice is functorial.

We must show that this choice is strictly stable. By definition, for  $\sigma \co \Delta \rightarrow \Gamma$,
\begin{align*}
(\Sigma_A B , \Sigma_A B [-]) [\sigma] &= ((\Sigma_{A} B)[\sigma], (\Sigma_{A} B)[\sigma][-]) \quad \quad \text{(by def.~of the cleavage of $(\rho^R, \chi^R)$)}\\
					      &= (\Sigma_{A[\sigma]} B[\sigma], (\Sigma_{A} B)[\sigma][-]) \quad \quad \text{(by def.~of $\Sigma_A B [-] $)  \rlap{.}}
\end{align*}
It only remains to show that the local cleavages $(\Sigma_A B)[\sigma][-]$ and $(\Sigma_{A[\sigma]} B[\sigma]) [-]$ coincide, but this follows from the functoriality of the pseudo-stable choice of $\Sigma$-types in $(\rho, \chi)$. 

The construction for the case of dependent products or $\Pi$-types is completely analogous and hence omitted.

For $\mathsf{Id}$-types, note that the terms of a type $(A, A[-])$ in $(\catC^R, \rho^R, \chi^R)$ are
the same as the terms of $A$ in $(\catC, \rho, \chi)$. Let $(A, A[-])$ be an object in the fibre of $\rho^R$ over $\Gamma$, and consider sections $a, b \co \Gamma \to \Gamma.A$.  We need  an object in $\catE^R$ over $\Gamma.A.A$, which we 
denote as 
\[
\big( \mathsf{Id}_{A}(a, b),  \mathsf{Id}_A(a, b)[-] \big) 
\]
for brevity. The object $\mathsf{Id}_A(a, b)$ is obtained by applying the pseudo-stable choice of $\mathsf{Id}$-types in $(\rho, \chi)$ to $A$ and the sections $a, b \co \Gamma \to \Gamma.A$. 

To define the local normal cleavage $\mathsf{Id}_{A}(a, b)[-]$, consider $\sigma \co \Delta \rightarrow \Gamma$ in $\catC$. The local normal cleavage $A[-]$ gives a Cartesian arrow $\sigma^* \co A[\sigma] \rightarrow A$ over $\sigma$. Using the stable functorial choice of $\mathsf{Id}$-types of $(\rho, \chi)$, we get the Cartesian arrow 
\[
\mathsf{Id}_{\sigma^*}(a, b) \co \mathsf{Id}_{A[\sigma]}(a[\sigma], b[\sigma]) \rightarrow \mathsf{Id}_A(a, b) 
\] 
over $\sigma$. We then let  $\mathsf{Id}_{A}(a, b)[\sigma] := \mathsf{Id}_{A[\sigma]}(a[\sigma], b[\sigma])$, so that the Cartesian arrow to~$\mathsf{Id}_A(a, b)$ can be taken to be $\mathsf{Id}_{\sigma^*}(a, b)$. By pseudo-stability, if $\sigma = 1_A$ then 
\[
\mathsf{Id}_{A[\sigma]}(a[\sigma], b[\sigma]) = \mathsf{Id}_{A}(a, b) \, \quad \mathsf{Id}_{\sigma^*}(a, b) = 1_{\mathsf{Id}_{A}(a, b)} \rlap{,} 
\]
as required.  It remains to check that this choice is strictly stable. For $\sigma \co \Delta \rightarrow \Gamma$ and sections $a, b \co \Gamma \to \Gamma.A$, we must verify that:
\[
(\mathsf{Id}_A(a, b)[\sigma], \mathsf{Id}_A(a, b)[\sigma][-]) = (\mathsf{Id}_{A[\sigma]}(a[\sigma], b[\sigma]), \mathsf{Id}_{A[\sigma]}(a[\sigma], b[\sigma]) [-]).
\] 
By definition, $\mathsf{Id}_{A}(a, b)[\sigma]$ is given by the local normal cleavage $\mathsf{Id}_A[-]$ applied to the arrow $\sigma^*$, which  is $\mathsf{Id}_{A[\sigma]}(a[\sigma], b[\sigma])$. By the functoriality of the pseudo-stable choice of $\mathsf{Id}$-types, the local cleavages $\mathsf{Id}_A(a, b)[\sigma][-]$ and $\mathsf{Id}_{A[\sigma]}(a[\sigma], b[\sigma]) [-]$ coincide.
\end{proof}

In the next sections we show how to construct comprehension categories with pseudo-stable choices of $\Sigma$-, $\Pi$- and $\mathsf{Id}$-types. For $\mathsf{Id}$-types, the structure that arises more naturally 
in example corresponds to the so-called  variable-based formulation of $\mathsf{Id}$-types~\cite[Table~3]{Gambino_2008}, which is known to be equivalent to the usual formulation. We  define this variant in \cref{defn:choiceID} and then show that it is equivalent to the notion of~\cref{defn:idp}. 

\begin{definition} \label{defn:choiceID}
A \emph{choice of variable-based $\mathsf{Id}$-types} on a comprehension category $(\catC, \rho, \chi)$ consists of an operation that assigns a tuple $(\mathsf{Id}_A, r_A, j_A)$, to each object $A$ in the fibre over some~$\Gamma \in \catC$,  where:
\begin{enumerate}
\item $\mathsf{Id}_A$ is an object in the fibre over $\Gamma.A.A$,
\item $r_A$ is a section of $\mathsf{Id}_A$ over the diagonal morphism $\delta_A$, giving a factorisation 
\[
\xymatrix{
&& \Gamma.A.A.\mathsf{Id}_A \ar[d]^{\chi_{\mathsf{Id}_A}} \\
\Gamma.A \ar[rr]_-{\delta_A} \ar[urr]^{r_A} && \Gamma.A.A \rlap{,}
}
\]

\item \label{item:j} $j_A$ is an operation that takes a pair $(C, t)$, where $C$ is an object $C$ in the slice over $\Gamma.A.A.\mathsf{Id}_A$ and $t$  is a section of $C$ over $r_A$, as in the diagram of solid arrows
\[
\xymatrix{
\Gamma.A \ar[d]_{r_A} \ar[rr]^{t} && \Gamma.A.A.\mathsf{Id}_A.C \ar[d]^{\chi_C} \\
\Gamma.A.A.\mathsf{Id}_A \ar@{.>}[urr]|{j_A(C, t)} \ar@{=}[rr] && \Gamma.A.A.\mathsf{Id}_A  \rlap{,}
}
\]
to a section $j_A(C, t)$ of $C$ (shown as the dotted arrow) making both triangles commute. 
\end{enumerate}
We will refer to a choice of variable-based $\mathsf{Id}$-types by $(\mathsf{Id}_{v.b.}, r, j)$.
\end{definition}

\begin{definition} \label{defn:psid}
A choice of variable-based $\mathsf{Id}$-types $(\mathsf{Id}_{v. b.}, r, j)$ in a comprehension category is said to be \emph{pseudo-stable} if for any Cartesian arrow $f \co B \rightarrow A$ over a morphism $\sigma \co \Delta \rightarrow \Gamma$ in the base, the following conditions are satisfied.
\begin{enumerate}
\item There is a Cartesian arrow $\mathsf{Id}_f \co \mathsf{Id}_{B} \rightarrow \mathsf{Id}_A$ over the canonical morphism $\delta_f \co \Delta.B.B \rightarrow \Gamma.A.A$, and the assignment $f \mapsto \mathsf{Id}_f$ is functorial, i.e. $\mathsf{Id}_{1_A} = 1_{\mathsf{Id}_A}$ and $\mathsf{Id}_{f \circ g} = \mathsf{Id}_{f} \circ \mathsf{Id}_{g}$.
\item The following diagram commutes:
\[
\xymatrix{
\Delta.B \ar[rr]^{f} \ar[d]_{r_{B}} && \Gamma.A \ar[d]^{r_A} \\
\Delta.B.B.\mathsf{Id}_{B} \ar[rr]_{\mathsf{Id}_f} && \Gamma.A.A.\mathsf{Id}_A
}
\]
\item For any pair $(C, t)$ as in $(3)$ of \cref{defn:choiceID} and for any Cartesian $h \co C' \rightarrow C$ over $\mathsf{Id}_f$, we can construct a pair $(C', t')$ by pulling back $t$ along $h$ appropriately. We require that the following diagram commutes:
\[
\xymatrix{
\Delta.B.B.\mathsf{Id}_{B} \ar[d]_{j_{B}(C', t')} \ar[rrr]^{\mathsf{Id}_f} &&& \Gamma.A.A.\mathsf{Id}_A \ar[d]^{j_A(C,t)} \\
\Delta.B.B.\mathsf{Id}_{B}.C' \ar[rrr]_{h} &&& \Gamma.A.A.\mathsf{Id}_A.C \rlap{,}
}
\]
where the lower horizontal arrow is the (comprehension of the) Cartesian arrow $h \co C' \rightarrow C$.
\end{enumerate}
\end{definition}


\begin{proposition} \label{prop:psequiv}
A  cloven comprehension category $(\catC, \rho, \chi)$ is equipped with a pseudo-stable choice of $\mathsf{Id}$-types if and only if it is equipped with a pseudo-stable choice of variable-based $\mathsf{Id}$-types.
\end{proposition}
\begin{proof}
Suppose we have a pseudo-stable choice of $\mathsf{Id}$-types. To construct a pseudo-stable choice of variable-based $\mathsf{Id}$-types, consider an object $A$ over $\Gamma$, we define $\mathsf{Id}_A := \mathsf{Id}_A(x, y)$ over $\Gamma.A.A$ using the canonical variables $x, y \co \Gamma.A.A \to \Gamma.A.A.A$. The operations $r$ and $j$ are given just as in \cref{defn:pseudostable}.

For the converse, assume a pseudo-stable choice of variable-based $\mathsf{Id}$-types. Consider an object $A$ over $\Gamma$ and sections $a, b \co \Gamma \to \Gamma.A$. We have $\mathsf{Id}_A$ over $\Gamma.A.A$, and thus, we can define $\mathsf{Id}_A(a, b)$ to be the reindexing of $Id_A$ along $(a, b)\co \Gamma \to \Gamma.A.A$, as in
\begin{equation}
\label{equ:idab}
\begin{gathered}
\xymatrix{
\mathsf{Id}_A(a, b) \ar[r] \ar@{.}[d] \drpullback & \mathsf{Id}_A \ar@{.}[d] \\
\Gamma \ar[r]_{(a, b)} & \Gamma.A.A \rlap{.} } 
\end{gathered}
\end{equation}
We can  extend this definition to provide the rest of the data in part~$(1)$ of \cref{defn:pseudostable}. Consider $f \co B \to A$ Cartesian over $\sigma \co \Delta \to \Gamma$, then we have
\[
\xymatrix{
\mathsf{Id}_{A[\sigma]}(a[\sigma], b[\sigma])  \ar@{.}[dd] \ar@{.>}[dr]^{\mathsf{Id}_f(a, b)} \ar[rr] && \mathsf{Id}_B \ar[dr]^{\mathsf{Id}_f} \ar@{.>}[dd]|(0.49)\hole|(0.30){} & \\
& \mathsf{Id}_A(a, b) \ar[rr] \ar@{.}[dd] && \mathsf{Id}_A \ar@{.}[dd] \\
\Delta \ar[rr]|(0.57)\hole_(0.75){(a[\sigma], b[\sigma])} \ar[dr]_{\sigma} && \Delta.B.B \ar[dr] & \\
& \Gamma \ar[rr]_{(a, b)} && \Gamma.A.A \rlap{.}
}
\]
where $\mathsf{Id}_f(a, b) \co \mathsf{Id}_{B}(a[\sigma], b[\sigma]) \rightarrow \mathsf{Id}_A(a, b)$ is given uniquely by the universal property of the square in~\eqref{equ:idab}. If $f$ (and $\sigma$) are identities, then by pseudo-stability of $\mathsf{Id}$-types, $Id_f \co Id_A \to Id_A$ is also the identity, and thus $\mathsf{Id}_f(a, b)$ must be the identity by the uniqueness property that characterises it. Similarly, the pseudo-stability of the variable-based $\mathsf{Id}$-types and the uniqueness property of $\mathsf{Id}_f(a, b)$ can be used to show that this operation preserves composition.

For the operations $r$ and $j$, we first consider the diagram
\[
\xymatrix{
\mathsf{Id}_A(x, y) \ar[rr]  \ar@{.}[d] \drpullback && \mathsf{Id}_A \ar[rr]^{Id_{\delta_{\chi_{A, A}}}}  \ar@{.}[d] \drpullback && \mathsf{Id}_A  \ar@{.}[d] \\
\Gamma.A.A \ar[rr]_{(x, y)} && \Gamma.A.A.A.A \ar[rr]_{\delta_{\chi_{A, A}}} && \Gamma.A.A \rlap{,}
}
\]
where the square on the right is obtained by the functoriality of the pseudo-stable choices of variable-based $\mathsf{Id}$-types to $\chi_{A, A} \co \Gamma.A.A \to \Gamma$. The square on the left is obtained by definition of $\mathsf{Id}_A(x, y)$ applied to the object $A$ weakened to $\Gamma.A.A$ and to the variables $x, y \co \Gamma.A.A \to \Gamma.A.A.A$. Both top horizontal arrows are Cartesian and the composition of the bottom two arrows equals the identity. Thus $\mathsf{Id}_A(x, y) \cong \mathsf{Id}_A$ as objects over $\Gamma.A.A$. We can then transport the operations $r$ and $j$ along this isomorphism.
\end{proof}

\begin{remark}
The reason for introducing two versions of identity types is that examples give rise more natually to
the variable-based $\mathsf{Id}$-types of \cref{defn:choiceID}, while the coherence theorem~\cref{thm:coht} 
is easier to prove constructively with the identity types of~\cref{defn:idp}. Indeed, if one works with variable-based $\mathsf{Id}$-types and tries to follow the argument used for $\Sigma$- and $\Pi$-types, the construction of the local cleavage  seems to require a case distiction on whether the argument is an identity or not to ensure functoriality. 
\end{remark}

Our goal in the reminder of the paper is to construct and study comprehension categories equipped
with pseudo-stable choices of  $\Sigma$-, $\Pi$- and $\mathsf{Id}$-types. By \cref{thm:coht}, these
will give rise to genuine models of Martin-L\"of type theory with $\Sigma$-, $\Pi$- and $\mathsf{Id}$-types.

\section{Algebraic weak factorisation systems} \label{sect:2-bis}

We review some of the basic theory on algebraic weak factorisation systems and on orthogonal categories of arrows~\cite{Bourke_2016, Bourke_2016_II,Garner_2009,Grandis_2006}. These will be the basis
for our definition of a type-theoretic algebraic weak factorisation system in~\cref{sect:3}. 

First of all, recall that a \emph{functorial factorisation} $(Q, L, R)$ on a category~$\catC$ consists of an operation that assigns to each arrow $f \co X \rightarrow Y$ a factorisation
\[
X \xrightarrow{Lf} Qf \xrightarrow{Rf} Y 
\] 
functorially in $f$. The induced endofunctors $L, R \co \catC^{\rightarrow} \rightarrow \catC^{\rightarrow}$ are canonically copointed and pointed respectively; that is, there are a counit $\epsilon \co L \rightarrow 1$ and a unit $\eta \co 1 \rightarrow R$.
We denote the category of $(R, \eta)$-algebras as $\Rmaps$ and the category of $(L , \epsilon)$-coalgebras by $\Lmaps$, and refer to their objects also as $R$-maps and $L$-maps, respectively. There are faithful, but not full, forgetful functors  to the arrow category, 
\[
\Lmaps \rightarrow \catC^{\rightarrow} \quad \text{and} \quad \Rmaps \rightarrow \catC^{\rightarrow} \rlap{.}
\]



Let $(g, \lambda) \co A \rightarrow B$ be an $L$-map, $(f, s) \co X \rightarrow Y$ an $R$-map  and $(h, k) \co g \rightarrow f$ a morphism in the arrow category. Then we can construct a  filler for the square $(h, k)$,
which is given by $j \co= s \cdot Q(h, k) \cdot \lambda$, where $Q(h, k) \co Qg \rightarrow Qf$ is the map obtained 
by applying the functorial factorisation to $(h, k)$. These  fillers satisfy naturality conditions with respect to morphism of $L$-maps and $R$-maps. 


%
\begin{definition} \label{defn:awfs}
An \emph{algebraic weak factorisation system} (\emph{awfs} for short) on a category~$\catC$ consists of the following data:
\begin{enumerate}
\item a functorial factorisation $(Q, L, R)$ on $\catC$,
\item an extension of the pointed endofunctor $(R, \eta)$ to a monad $(R, \eta, \mu)$,
\item an extension of the copointed endofuctor $(L , \epsilon)$ to a comonad $(L , \epsilon, \delta)$,
\item \label{item:DL} the canonical map $\Delta \co LR \rightarrow RL$ defined using the monad and comonad structure is a distributive law.
\end{enumerate}
\end{definition}

We refer to an awfs as in \cref{defn:awfs} just as $(L, R)$. 
Item (\ref{item:DL}) of \cref{defn:awfs} is a technical requirement which plays no role in this work. Given
an awfs as \cref{defn:awfs}, we have also the category of algebras for the monad $(R, \eta, \mu)$,
which we denote $\Ralg$, and the category of coalgebras for the comonad $(L , \epsilon, \delta)$,
which we denote $\Lalg$. We refer to the objects of $\Ralg$ and $\Lalg$  as $R$-algebras and~$L$-coalgebras, respectively. There are full and faithful functors $\Ralg \hookrightarrow \Rmaps $ and $\Lalg \hookrightarrow \Lmaps$. 

\begin{remark} \label{thm:vert-comp-pbk-awfs}
The category $\Ralg$ (and also $\Rmaps$) is closed under `vertical' composition; that is if $(f , s) \co X \rightarrow Y$ and $(f', s') \co Y \rightarrow Z$ are $R$-algebras then there is a canonical $R$-algebra structure $s' \cdot s$ on the composite $f' \cdot f$. In fact, finding such a vertical composition operation provides a complete characterisations of the awfs \cite[Theorem~4.15]{Barthel_2013}. 

Also, for an $R$-algebra $(f, s)$ and a pullback square $(h , k) \co f' \rightarrow f$ then, there exists a unique $R$-algebra structure $s'$ on $f'$ making $(h ,k)$ a morphism of $R$-algebras. The same result holds for $R$-maps. 
\end{remark}

We recall some notions regarding categories of arrows and of orthogonality in the setting of awfs's. By a \emph{category of arrows} over $\icat{C}$ we mean a functor $u \co \lcat{J} \rightarrow \icat{C}^\rightarrow$ where~$\lcat{J}$ is a category.  A \emph{right $\lcat{J}$-map} consists of a pair $(f , \theta)$ where $f \co X \rightarrow Y$ is an arrow of $\icat{C}$ and $\theta$ is a right lifting operation against $\lcat{J}$, {i.e.}~$\theta$ assigns a filler $\theta(i)$ to each commutative square of the form $(l,m) \co u_i \rightarrow f$, with $i \in \lcat{J}$. These fillers, in addition, are compatible with the arrows in $\lcat{J}$ in the evident way. 

Given a pair of right $\Jay$-maps $(f, \theta)$ and $(f', \theta')$, a \emph{right $\Jay$-map morphism} consists of a square $(\alpha, \beta) \co f \rightarrow f'$ such that for every $i \in \Jay$, the triangle created by the corresponding choices of diagonal fillers commute.
Given  a category of arrows $u \co \Jay \rightarrow \catC^{\rightarrow}$, we define the category $\Jay^{\boxslash}$ consisting of right $\Jay$-maps $(f, \theta)$ together with the corresponding morphisms. There is a functor
$
u^{\boxslash} \co \Jay^{\boxslash} \rightarrow \catC^{\rightarrow}
$
forgetting the lifting structure. It can be shown that this operation defines a contravariant functor denoted by $(-)^{\boxslash}$. In a completely analogous manner, we can define the concepts of {left $\Jay$-map} and {left $\Jay$-map morphism}, and obtain a dual functor $^{\boxslash}\!(-)$. This data constitutes the \emph{orthogonality adjunction}, which generalises the classical Galois connection between orthogonal classes of maps:
\begin{equation}
\label{equ:orth-adj}
\xymatrix{
\CAT / \catC^{\rightarrow} \ar@{}[rr]|{\bot} \ar@/^0.7pc/[rr]^{^{\boxslash}\!(-)} && (\CAT / \catC^{\rightarrow})^{op} \ar@/^0.7pc/[ll]^{(-)^{\boxslash}} \rlap{.}
}
\end{equation}
The next proposition~\cite[Lemma 1]{Bourke_2016} relates awfs and orthogonal categories of arrows.

\begin{proposition} \label{prop:liftings}
Let $(L, R)$ be an awfs on $\catC$. Then, there are lifting functors over $\catC^{\rightarrow}$ as shown in the following commutative diagram:
\[
\xymatrix{
\Ralg \ar[rr]^{\lift} \ar[d] && (\Lalg)^{\boxslash} \\
\Rmaps \ar[rru]^{\lift}_{\cong} \ar[rr]_{\lift} && (\Lmaps)^{\boxslash} \ar[u]
}
\]
All functors are full and faithful and the diagonal one is an isomorphism. There is also a functor $(\Lmaps)^{\boxslash} \rightarrow \Rmaps$, which is not an equivalence in general. \qed
\end{proposition} 

We say that an awfs $(L, R)$ is \emph{algebraically-free} on a category of arrows $\Jay$ if there is a functor $\eta \co \Jay \rightarrow \Lalg$ over $\catC^{\rightarrow}$, such that the composition
\[
\xymatrix{
\Ralg \ar[r]^(.40){\lift} & (\Lalg)^{\boxslash} \ar[r]^(.65){\eta^{\boxslash}} & (\Jay)^{\boxslash}
}
\]
is an isomorphism of categories, cf.~\cite[Theorem~4.4]{Garner_2009}.
The following result regarding algebraically-free awfs is implicit in the literature (cf. \cite[Theorem~6.9]{Gambino_2017} for example).

\begin{proposition} \label{prop:baf}
If $(L, R)$ is algebraically-free on some category of arrows $\Jay$, then there are 
functors back-and-forth $\Rmaps \leftrightarrow \Ralg$ over  $\catC^{\rightarrow}$.
\qed 
\end{proposition} 

This proposition shows that when working with an algebraically-free awfs $(L, R)$ (as will be the case in \cref{sect:UF}), any construction made using $R$-maps can be functorially transported to a construction using $R$-algebras, and viceversa.     

\section{Type-theoretic awfs's} \label{sect:3}

In this section we introduce the notion of a type-theoretic awfs. We then 
show how a type-theoretic awfs induces a comprehension category structure equipped with pseudo-stable choices of $\Sigma$-, $\Pi$- and $\mathsf{Id}$-types. We begin by making the connection between awfs and comprehension categories. 

\begin{lemma} \label{lemma:RmapGF}
Let $(L,R)$ be an awfs over $\catC$. The functor $\Rmaps \rightarrow \catC$ mapping an $R$-map $(f, s)$ to $cod(f)$ is a Grothendieck fibration. Moreover, the Cartesian arrows are the morphisms of $R$-maps whose underlying square is a pullback square. \qed
\end{lemma}

\begin{proposition} \label{prop:AWFSCC}
Let $(L, R)$  be an awfs on a category $\catC$. Then there is a comprehension category
\[
\xymatrix{
\Rmaps \ar[rr]^U \ar[dr] && \catC^{\rightarrow} \ar[dl]^{cod} \\
& \catC \, , &
}
\]
where $U$ is the evident forgetful functor. \qed
\end{proposition}


Results analogous to~\cref{lemma:RmapGF} and \cref{prop:AWFSCC} hold for also for~$\Ralg$.
Next, we study additional logical structure on the comprehension category induced by an awfs. 

\begin{proposition} \label{prop:sigmaAWFS}
Let $(L, R)$ be an awfs on $\catC$. Then the comprehension category induced by $(L, R)$ is equipped with a pseudo-stable choice of $\Sigma$-types.
\end{proposition}

\begin{proof} Let $(f, s) \co X \to \Gamma$ and $(g, t) \co Y \to X$ be in~$\Rmaps$. 
The pullback functor along~$f  \colon X \to \Gamma$ has a left adjoint $\Sigma_f \co  \catC/X \rightarrow  \catC/\Gamma$, which is given by composition. By  \cref{thm:vert-comp-pbk-awfs}, this functor lifts to slices of $\Rmaps$, as follows
\[
\xymatrix{
\Rmaps/X \ar[r]^-{\Sigma_f}  \ar[d] & \Rmaps/\Gamma  \ar[d] \\ 
\catC/X \ar[r]_-{\Sigma_f} & \catC/\Gamma \rlap{.} }
\]
For the formation rule, we define $\Sigma_f g \co= \Sigma_f (g) = f \circ g\co Y \to \Gamma$.
For the introduction rule, we define $pair_{f,g} \co Y \rightarrow Y$ over $f$, 
\[
\xymatrix@C=1.5cm{
Y \ar[r]^{pair_{f,g}} \ar[d]_{g} & Y \ar[d]^{\Sigma_f g} \\
X \ar[r]_{f} & \Gamma
}
\]
by letting $pair_{f,g} \co= 1_{Y}$. Finally, for the elimination rule, let $C$ be over $\Sigma_f g$ and let $t$ be a section of $C$ over $pair_{f, g}$, we define~$sp_{f,g}(C,t) \co= t$. The computation rule holds trivally, thus giving rise to a choice of~$\Sigma$-types. Stable functoriality and coherence of elimination terms also follow easily, the crucial observation is that vertical composition of $R$-maps plays nicely with the horizontal categorical structure (see \cite[Section~2.8]{Bourke_2016}). 
\end{proof}

The case of $\Pi$-types requires the following property.

\begin{definition} \label{defn:expon}
An awfs $(L,R)$ on  $\catC$ satisfies the {\em exponentiability property} if  for any~$g \co Z \rightarrow Y$,
$f \co Y \rightarrow X$ in the image of $\Rmaps \rightarrow \catC^{\rightarrow}$, the exponential $\Pi_fg \in \catC/X$ exists. 
\end{definition}

Clearly, any awfs in a locally Cartesian closed category satisfies the exponentiability property. We need something more than mere exponentiability, namely a way to coherently lift an exponential 
from $\catC/X$ to  $\Rmaps/X$. For this reason we recall the following  notion from~\cite{Gambino_2017}.

\begin{definition} \label{defn:FFstructure}
Let $(L, R)$ be an awfs on a category $\catC$. A {\em functorial Frobenius structure} is given by a lift of the pullback functor as shown:
\[
\xymatrix{
\Rmaps \times_{\catC} \Lmaps \ar@{..>}[rr]^{\tilde{PB}} \ar[d] && \Lmaps \ar[d] \\
\catC^{\rightarrow} \times_{\catC} \catC^{\rightarrow} \ar[rr]_{PB} && \catC^{\rightarrow} \, ,
}
\]
where $PB(f, g)$ denotes the pullback of $g$ along $f$.
\end{definition}


\begin{proposition} \label{prop:piAWFS}
Consider an awfs $(L, R)$ on $\catC$ satisfying the exponentiability property and equipped with a functorial Frobenius structure. Then the comprehension category induced by $(L, R)$ has a pseudo-stable choice of $\Pi$-types.
\end{proposition}

\begin{proof}
Consider  $(f, s) \co X \rightarrow \Gamma$ in $\Rmaps$. By the exponentiability property,
we have a pushforward functor $\Pi_f \co \category{R}/X \to \catC/\Gamma$ (here $\category{R}/X$ denotes the slice category whose objects are arrows $g \co Y \to X$ that can be equipped with an $R$-map structure) and, by~\cite[Proposition~6.5 and Proposition~6.7]{Gambino_2017} this lifts to a functor
\[
\Pi_f\co (\Lmaps)^{\boxslash}/X \rightarrow (\Lmaps)^{\boxslash}/\Gamma \, .
\]
Using the functors $\Rmaps \leftrightarrow \Lmaps^\boxslash$ of \cref{prop:liftings} 
we can find a lift of $\Pi_A$ as follows:
\[
\xymatrix{
 \Rmaps/{X} \ar[r]^{\Pi_f} \ar[d] &  \Rmaps / {\Gamma} \ar[d] \\
 \catC / {X} \ar[r]_{\Pi_f} & \catC /{\Gamma} \rlap{.} }
\]

Consider an $R$-map $(g ,t) \co Y \to X$. For
the formation rule, we apply $\Pi_f$ to obtain an $R$-map $\Pi_f g \co Y \rightarrow \Gamma$.
For the introduction rule, we define the operation $\lambda$; consider a section~$t$ of~$g$, this is an arrow $t \co 1_{X} \rightarrow g$ in $\catC/X$ and since $f^*(1_{\Gamma}) \cong 1_{X}$ this is the same thing as an arrow $t \co f^*(1_{\Gamma}) \rightarrow g$. Taking the transpose yields a map~$\lambda(t) \co 1_{\Gamma} \rightarrow \Pi_f g$,  as required. 
For the elimination rule we need to provide an arrow $app_{f,g} \co f^*(\Pi_f g) \rightarrow g$. We can take $app_{f,g}$ to be the counit of the adjunction, 
$app_{f,g}\co= \epsilon_{g} \co f^*(\Pi_f g) \rightarrow g$. The computation rule follows easily from the bijection~$\lambda \co \catC /X [1_{X}, g] \to \catC /\Gamma[1_{\Gamma}, \Pi_f g]$.

It only remains to show the assignment $(\Gamma, f, g) \mapsto (\Pi_f g, \lambda, app)$ is pseudo-stable.
This is a diagram-chasing argument, which relies on the fact that, for a Cartesian square of the form
\[
\xymatrix{
X' \drpullback \ar[r]^{\tau} \ar[d]_{f'} & X \ar[d]^f \\
\Delta \ar[r]_{\sigma} & \Gamma }
\]
the Beck-Chevalley isomorphism $BC \co \Delta_\sigma \Pi_f \rightarrow \Pi_{f'} \Delta_\tau$
(where we write $\Delta_\sigma$ and $\Delta_\tau$ for the pullback functors along $\sigma$ and 
$\tau$, respectively) lifts to an isomorphism of $R$-maps by  \cite[Proposition~6.7]{Gambino_2017}. We leave the details to the readers.
\end{proof}

\begin{remark}
As shown  above, an awfs equipped with a functorial Frobenius structure implies the existence of lifts $\Sigma_f\co \Rmaps/X \rightarrow \Rmaps/\Gamma$ and $\Pi_f\co \Rmaps/X \rightarrow \Rmaps/\Gamma$ of the composition and pushforward functor, respectively, for each $R$-map~$(f,s)$. However, the underlying adjunctions need not lift to $\Rmaps$. Fortunately, this is not necessary for the construction of pseudo-stable choices of $\Sigma$- and $\Pi$-types since we only need the universal property at the level of the underlying category.
\end{remark}

The case for intensional identity types is more complicated. Here the extra algebraic structure is essential, it will allow us to keep track of the necessary information needed to coherently produce the `elimination terms' (i.e. the fillers $j$ of \cref{item:j} from \cref{defn:choiceID}) for the choice of $\mathsf{Id}$-types.
To address this issue, recall that a \emph{functorial factorisation of the diagonal} is a functor $\lcat{P} \co \catC^{\rightarrow} \rightarrow \catC^{\rightarrow} \times_\catC \catC^\rightarrow$ that acts on a map $f \co X \rightarrow Y$ as 
\[
f \mapsto (X \xrightarrow{r_f} PX \xrightarrow{\rho_f} X \times_{Y} X) \rlap{,}
\]  
such that the composition $\rho_f \cdot r_f$ equals the diagonal morphism  $\delta_f \co X \rightarrow X \times_Y X$. We say that a functorial factorisation of the diagonal is \emph{stable} if the square $\rho_{(h, k)} \co \rho_{f'} \rightarrow \rho_f $
is  Cartesian when $(h, k) \co f' \rightarrow f$ is so. 
We denote a (stable) functorial factorisation of the diagonal by  $\lcat{P} = \abra{r, \rho}$, where~$r, \rho \co  \catC^{\rightarrow} \rightarrow  \catC^{\rightarrow}$ are the induced functors from the two legs of the factorisation respectively. The following notion was first described in \cite[Definition~3.3.3]{Berg:2010aa}.

\begin{definition} \label{defn:SFPO}
Let $(L, R)$ be an awfs on $\catC$. A {\em stable functorial choice of path objects} (or {\em sfpo} for conciseness) consists of a lift of a stable functorial factorisation of the diagonal $\lcat{P}$ as shown in the following diagram:
\[
\xymatrix{
\Rmaps \ar[rr]^-{\lcat{P}} \ar[d] && \Lmaps \times_{\catC} \Rmaps \ar[d] \\
\catC^{\rightarrow} \ar[rr]_-{\lcat{P}} && \catC^{\rightarrow} \times_\catC \catC^\rightarrow \rlap{.}
}
\]
\end{definition}

\begin{proposition} \label{prop:SFPO->PSID}
Let $(L, R)$ be an awfs equipped with a sfpo of the form $\lcat{P} = \abra{r, \rho}$. Then $(L, R)$ is equipped with the structure of a pseudo-stable choice of $\mathsf{Id}$-types.
\end{proposition}

\begin{proof} We first construct a pseudo-stable choice $(\mathsf{Id}_{v.b.}, r, j)$ of variable-based $\mathsf{Id}$-types (see \cref{defn:choiceID}) and then apply \cref{prop:psequiv} in order to obtain a pseudo-stable choice of $\mathsf{Id}$-types (see \cref{defn:pseudostable}). 

The choices for $\mathsf{Id}$ and $r$ are canonically given by the stable functorial choice of path objects. These satisfy the coherence properties of \cref{defn:psid}. Since the maps $r_f$ are equipped with an $L$-map structure, we have lifts against $R$-maps. Using this, we obtain a choice of canonical elimination terms (i.e. $j$-terms). 

We are left to verify that this choice is coherent. For this, it is sufficient to show that given a Cartesian morphism of $R$-maps $(h , k) \co f' \rightarrow f$, a $R$-map $q \co C \rightarrow PX$, if the diagram on the left of~\eqref{equ:fol-fig} commutes, then so does the one on the right.
\begin{equation}
\label{equ:fol-fig}
\begin{gathered}
\xymatrix{
X \ar[d]_{r_f} \ar[r]^{d} & C \ar[d]^{q} && C^* \ar[r]^{P(h ,k)^*}  & C\\
PX \ar@{=}[r]  & PX && PX' \ar[r]_{P(h, k)} \ar[u]^{j(q^*)} & PX \ar[u]_{j(d)} \rlap{,}
} 
\end{gathered}
\end{equation}
where $q^* \co C^* \rightarrow PX'$ is defined as the pullback of $q$ along $P(h,k)$. The arrows denoted by $j$ are the canonical choices of lifts. The arrow $d^*$ is the pullback of $d$ along $P(h, k)$, i.e. it is defined to be the unique arrow $d^* \co X' \rightarrow C^*$ such that:

\begin{align} \label{equation:d^*}
 q^* \circ d^* = r_{f'} \quad \text{and} \quad P(h, k)^* \circ d^* = d \circ h. 
\end{align}
We split the problem into two. First, consider the following diagram equipped with the corresponding canonical lifts:
\[
\xymatrix{
X' \ar[r]^{d^*} \ar[d]_{r_{f'}} & C^* \ar[d]^(0.6){q^*} \ar[r]^{P(h,k)^*} & C \ar[d]^{q} \\
PX' \ar@{=}[r] \ar@{.>}[ur]|{j(d^*)} \ar@{.>}[urr]|(.70)j & PX' \ar[r]_{P(h, k)} & PX \rlap{,}
}
\]
Note that $j = P(h, k)^* \circ j(d^*)$ since the Cartesian square $q^* \rightarrow q$ is a morphism of~$R$-maps. Now consider the following lifting problem
\[
\xymatrix{
X' \ar[r]^{h} \ar[d]_{r_{f'}} & X \ar[d]_(0.3){r_{f}} \ar[r]^{d} & C \ar[d]^{q} \\
PX' \ar[r]_{P(h, k)} \ar@{.>}[urr]|(.30){j'} & PX \ar@{.>}[ur]|{j(d)} \ar@{=}[r] & PX \rlap{.}
}
\]
Once more, $j' = j(d) \circ P(h, k)$ since the square $r_{f'} \rightarrow r_f$ is morphism of $L$-maps. 
Finally,~\eqref{equation:d^*} tells us that the outer squares of the two previous diagrams are equal, implying that they have the same lift $j = j'$. Thus, $P(h , k)^* \circ j(d^*) = j(d) \circ P(h, k)$ as needed.
\end{proof} 

Type-theoretic awfs's, defined below, collect the structure that we discussed so far. 

\begin{definition} \label{defn:algmod} Let $\catC$ be a category. 
A {\em type-theoretic awfs} on $\catC$ consists of the following data:
\begin{enumerate}
\item  an awfs $(L, R)$ on $\catC$  satisfying the exponentiability property,
\item a functorial Frobenius structure on $(L, R)$,
\item a stable functorial choice of path objects on $(L, R)$, 
\end{enumerate}
\end{definition}

The following theorem summarises our results obtained so far in this section.

\begin{theorem} \label{thm:AMtoCC}
 Let $(L, R)$ be an awfs on a category $\catC$ with the structure of a type-theoretic awfs.  Then the comprehension category induced by $(L, R)$ is equipped with pseudo-stable choices of $\Sigma$-, $\Pi$- and $\mathsf{Id}$-types. 
\end{theorem}
\begin{proof}
Apply \cref{prop:sigmaAWFS}, \cref{prop:piAWFS} and \cref{prop:SFPO->PSID}.
\end{proof}


We conclude the section describing how type-theoretic awfs's give rise to models of  type theory.

\begin{theorem} \label{thm:combine}
Let $(L, R)$ be an awfs on a category $\catC$ with the structure of a type-theoretic awfs. Then the right adjoint splitting of the comprehension category associated to $(L, R)$ is equipped with strictly stable choices of $\Sigma$-, $\Pi$- and $\mathsf{Id}$-types. 
\end{theorem}

\begin{proof} Combine \cref{thm:coht} and  \cref{thm:AMtoCC}.
\end{proof} 

\section{Revisiting the groupoid model} \label{sect:4}

The aim of this section is to provide a first example of a type-theoretic awfs by revisiting the original Hofmann-Streicher model \cite{Hofmann_1998} on the category of groupoids. Explicitly, we construct a type-theoretic awfs $(C_f, F)$ on the category $\GRP$ of groupoids and functors. 

Consider $f \co X \rightarrow Y$ a functor between groupoids. The comma category of $f$, denoted by $\commacat{f}$, has as objects tuples $(a, b, p)$ with $a \in X$, $b \in Y$ and $p \co b \rightarrow fa$. We have that $\commacat{f}$ is again a groupoid, and moreover the construction is functorial: $\commacat{(-)} \co \GRP^{\rightarrow} \rightarrow \GRP$. This forms the middle part of a functorial factorisation assigning  
\[
\xymatrix{
X \ar[r]^-{C_t f} & \commacat{f} \ar[r]^-{F f} & Y
}
\]
to $f \co X \rightarrow Y$, where $C_t f (a) = (a, fa, 1_{fa})$ and $Ff(a, b, p) = b$. 

\begin{proposition} \label{prop:grpawfs}
The functorial factorisation $(\commacat{(-)}, C_t, F)$ is an algebraic weak factorisation system on $\GRP$. The $C_t$-maps are the strong deformation retractions, while the $F$-maps are the normal isofibrations. 
\end{proposition}

\begin{proof}
We start by examining the structures of the $C_t$-maps and the $F$-maps. We know that an $F$-map structure on a map $f \co X \rightarrow Y$ corresponds to a lift $s$ as shown on the diagram on the left below:
\[
\xymatrix{
X \ar[d]_{C_t f} \ar@{=}[r] & X \ar[d]^f  && A \ar[d]_{g} \ar[r]^{C_t g} & \commacat{g} \ar[d]^{Fg} \\
\commacat{f} \ar[r]_{F f} \ar@{.>}[ur]|s & Y && B \ar@{=}[r] \ar@{.>}[ur]|{\lambda} & B \rlap{.}
 }
\]
A closer analysis will show that $s$ equips $f \co X \rightarrow Y$ with the structure of a normal isofibration. An $L$-map structure on $g \co A \rightarrow B$, is given by a lift $\lambda$ as shown on the diagram on the right of the previous figure. The structure obtained from such a lift $\lambda$ can be decomposed as $\lambda(b) = (\lambda_1(b) , b , \lambda_2(b))$ where $\lambda_1 \co B \rightarrow A$ corresponds to a retraction of $g$ and $\lambda_2 \co 1_B \rightarrow g \circ \lambda_1$ corresponds to a natural transformation constant on the image of $f$. This information corresponds to the structure of a strong deformation retraction.

We construct the corresponding structures of a comonad and a monad for $C_t$ and $F$ respectively. We provide a brief description and leave the details to the reader. The comultiplication $\delta_f \co \commacat{f} \rightarrow \commacat{C_t f}$ for $C_t$  is defined by letting
\[
\delta_f \co (a, b , p) \mapsto (a, (a, b, p), (1_a, p) \co (a, b, p) \rightarrow (a, Fa, 1_{fa})) \rlap{.}
\]
Similarly, the endofunctor $F$ has a multiplication $\mu_f \co \commacat{Ff} \rightarrow \commacat{f}$ given by
\[
\mu_f \co ((a, b , p), \tilde{b}, \tilde{p} \co \tilde{b} \rightarrow b) \mapsto (a, \tilde{b}, p \circ \tilde{p}) \, .\qedhere
\]
\end{proof}


\begin{remark}
The identification of the $F$-maps with normal isofibrations implies that the category $\GRP$ satisfies the exponentiability condition (see \cref{defn:expon}) with respect to the awfs $(C_t, F)$ since isofibrations can be exponentiated \cite{Conduche_1972}, even if $\GRP$ is not locally Cartesian closed. The  $F$-algebras can be identified with \emph{split isofibrations}. An extension of the theory considered here to $F$-algebras has been considered in~\cite{Woerkom}.
\end{remark}


\begin{proposition} \label{prop:GrpFro}
The awfs $(C_t, F)$ is equipped with a functorial Frobenius structure.
\end{proposition}

\begin{proof}
We show that pulling back a $C_t$-map along an $F$-map is uniformly a $C_t$-map. Consider $(g, \lambda) \co A \rightarrow Y$ a $C_t$-map and $(f, s) \co X \rightarrow Y$ an $F$-map. Let $g' \co A \times_Y X \rightarrow X$ be the pullback of $g$ along $f$. We define a $C_t$-map structure $\lambda'$ on $g'$ which, by \cref{prop:grpawfs}, corresponds to a strong deformation retraction $(g', \lambda_1', \lambda_2')$. Using that~$f$ corresponds to a normal isofibration, we can find for each point $x \in X$, a point $x' \in X$ and a lift $\lambda_2'(x)$ of $\lambda_2(fx)$, as in
\[
\xymatrix{
x \ar[rr]^{\lambda_2'(x)} \ar@{.}[d] && x' \ar@{.}[d] \\
fx \ar[rr]_{\lambda_2(fx)} && g\lambda_1fx \rlap{.}
}
\]
We define $\lambda_1'(x) = (\lambda_1(fx), x')$, the homotopy $\lambda_2' \co 1 \rightarrow g' \circ \lambda_1'$ is defined using the top arrow in the previous diagram.
\end{proof}

We turn our attention to identity types. The category $\GRP$ has a stable and functorial factorisation of the diagonal given on a map $f \co X \rightarrow Y$ by:
\[
\xymatrix{
X \ar[r]^{r_f} & Pf \ar[r]^{\rho_f} & X \times_Y X \rlap{,}
}
\]
where the objects of $Pf$ are tuples $(a, a', p)$ such that $p \co a \rightarrow a'$ is a morphism in $X$ over the identity, i.e. $fa = fa'$ and  $fp = 1_{fa}$. The map $r_f$ is given by $a \mapsto (a, a, 1_a)$ and the map $\rho_f$ is given by $(a, b, p) \mapsto (a, b)$. 

\begin{proposition} \label{prop:GrpWar}
The awfs $(C_t, F)$ is equipped with a stable and functorial choice of path objects.
\end{proposition}

\begin{proof} 
For an $F$-map $(f,s) \co X \rightarrow Y$, we need to uniformly provide a $C_t$-map structure to~$r_f$ and an $F$-map structure to~$\rho_f$. 
Let us define $\lambda_1 := t_f \co Pf \rightarrow X$ the canonical target map. We define the natural transformation $\lambda_2 \co 1_{Pf} \rightarrow r_f \circ t_f$ by  
\[
\lambda_2(a, a', p) := (p, 1_{a'}) \co (a, a', p) \rightarrow (a', a', 1_a') \rlap{.} 
\]
This corresponds to a strong deformation retraction structure on $r_f$.
An $F$-map structure on $\rho_f$ corresponds to a normal isofibration. Consider $(\alpha, \beta) \co (b, b') \rightarrow (a, a')$ in $X \times_Y X$ and an object $(a, a', p) \in Pf$ over $(a, a')$. We find the lift $(\alpha, \beta) \co (b, b', q) \rightarrow (a, a', p)$ by setting $q := \beta \circ p \circ \alpha^{-1} \co b \rightarrow b'$.
\end{proof}


\begin{theorem} \label{thm:Groupmain}
The awfs $(C_t, F)$ on the category $\GRP$ is equipped with the structure of a type-theoretic awfs. 
\end{theorem}

\begin{proof}
Apply \cref{prop:grpawfs}, \cref{prop:GrpFro} and \cref{prop:GrpWar}.
\end{proof}
 
By \cref{thm:combine} and \cref{thm:Groupmain} we obtain a version of the groupoid model of~\cite{Hofmann_1998}, using normal isofibrations instead of split fibrations and 
presented in terms of a split comprehension category rather than of a category with families~\cite{Dybjer_1996}. In our presentation, the connection to the homotopy theory of groupoids is made explicit thanks to the notion of a type-theoretic awfs. 
\section{Type-theoretic awfs from uniform fibrations} \label{sect:UF}

In this section we  investigate how to obtain type-theoretic awfs using the theory of uniform fibrations of \cite{Gambino_2017}. This provides a major source of examples of categories equipped with type-theoretic awfs, including some on simplicial and cubical sets. 

We begin by recalling the pushout-product construction \cite{Riehl:2013aa}. 
 Let us consider a Grothendieck topos $\catE$, the \emph{pushout-product bifunctor}
$- \hat{\times} -  \co \catE^{\rightarrow} \times \catE^{\rightarrow} \rightarrow \catE^{\rightarrow}$
is defined on a pair of arrows $f \co X \rightarrow Y$ and $g \co A \rightarrow B$ as the  dotted arrow 
in 
\[
\xymatrix{
X \times A \ar[rr]^{f \times A} \ar[d]_{X \times g} && Y \times A \ar[d] \ar@/^10pt/[rdd]^{Y \times g} & \\
X \times B \ar[rr] \ar@/_10pt/[rrrd]_{f \times B} && \drpushout (Y \times A) +_{X \times A} (X \times B) \ar@{.>}[dr]|{f \hat{\times} g} & \\
&&& Y \times B \, .
}
\]

An \emph{interval object} in $\catE$ consists a object $I$ together with two morphisms
$
\delta^0, \delta^1 \co \bot \rightarrow I
$
(where we write $\bot$ for the terminal object of $\catE$), respectively called the left and right \emph{endpoint inclusions};  these morphisms are required to be disjoint, \emph{i.e.}~the pullback of one along the other matches the initial object. We require the following additional structure. The \emph{connection operations} on $I$ are given by
$
c^k \co I \times I \rightarrow I
$
for $k \in \{0,1\}$, making the following diagrams commute:
\[
\xymatrix{
I \ar[r]^{\delta^k \times I} \ar[d]_{\epsilon} & I \times I \ar[d]^{c^k}   &     
I \ar[rr]^{\delta^{1 - k} \times I} \ar@{=}[drr] && I \times I \ar[d]^{c^k}   
\\
\bot \ar[r]_{\delta^k}  & I  \rlap{,}  &     
 && I \rlap{.}
}
\]
Connections correspond to special type of degeneracy maps that can be pictured as the two possible deformations of the square $I \times I$ into its diagonal by fixing one of the two endpoints.
With this in place, we proceed to describe the construction of uniform fibrations. Our starting point is the following definition. 

\begin{definition} \label{defn:suittops}
A {\em suitable topos} consists of a tuple $(\catE, I, \Em)$ where $\catE$ is a Grothendieck topos equipped with an interval object $I$ with connections and a class $\Em$ of arrows in $\catE$ satisfying the following conditions:
\begin{itemize}
\item[(M1)] the objects of $\Em$ are monomorphisms,
\item[(M2)] the inital map $\emptyset \rightarrow X$ is in $\Em$ for every $X \in \catE$,
\item[(M3)] the objects of $\Em$ are closed under pullback along any arrow in $\catE$,
\item[(M4)] the elements of $\Em$ are closed under pushout-product with the endpoint inclusions, \emph{i.e.}~for each $j \in \Em$, we have that $\delta^k \hat{\times} j \in \Em$.
\end{itemize}
The elements of $\Em$ are called \emph{generating monomorphisms}.
\end{definition}


Given a suitable topos $(\catE, I, \Em)$, we can consider $\Em$ as a category by taking Cartesian squares as arrows. Now, let us denote by $\Em_{\htimes}$ the category that has as objects maps of the form $\delta^k \htimes j$ with $j\in \Em$ and $k \in \{0,1\}$ and whose morphisms are given by squares of the form
$
\delta^k \htimes \sigma \co (\delta^k \htimes j') \rightarrow (\delta^k \htimes j)
$
induced by functoriality of the pushout-product applied to Cartesian squares $\sigma \co j' \rightarrow j$ between generating monomorphisms. We consider $\Em_{\htimes}$ to be a category of arrows by taking the inclusion into $\catE^{\rightarrow}$.

\begin{construction}
Let us consider a suitable topos $(\catE, I, \Em)$. The category of arrows of {\em trivial uniform fibrations}, denoted by 
\[
\category{TrivUniFib} \rightarrow \catE^{\rightarrow} \rlap{,} 
\]
is defined as the right orthogonal category of arrow to $\Em$, that is, $\category{TrivUniFib}\co= \Em^{\boxslash}$. Analogously, the category of arrows of {\em uniform fibrations}, denoted by
\[
\category{UniFib} \rightarrow \catE^{\rightarrow} \rlap{,}
\]
is defined as the right orthogonal category of arrow to $\Em_{\htimes}$, i.e. $\category{UniFib}\co= \Em_{\htimes}^{\boxslash}$.
\end{construction}

We construct awfs's of \emph{trivial uniform fibrations} and \emph{uniform fibrations}, even if we cannot apply Garner's small object argument directly because $\Em$ need  not be small.

\begin{lemma} \label{lemma:awfssmall}
Consider a suitable topos $(\catE, I, \Em)$ with a fixed dense small subcategory~$\catA$. Let us denote by $\Ai$ the full subcategory of $\Em$ spanned by those arrows in $\Em$ whose codomain lie in $\catA$. Similarly, denote by $\Ai_{\htimes}$ the full subcategory of $\Em_{\htimes}$ whose objects are pushout-product maps $\delta^k \htimes j$ with $j \in \Ai$. Then the following are satisfied:
\begin{enumerate}
\item  The right orthogonal functor of the inclusion $\mathsf{inc} \co \Ai \hookrightarrow \Em$ is an isomorphism \emph{i.e.}~$\mathsf{inc}^{\boxslash} \co \Em^{\boxslash} \cong \Ai^{\boxslash}$.
\item The right orthogonal functor of the inclusion $\mathsf{inc}_{\htimes} \co \Ai_{\htimes} \hookrightarrow \Em_{\htimes}$ is an isomorphism \emph{i.e.}~$\mathsf{inc}_{\htimes}^{\boxslash} \co \Em_{\htimes}^{\boxslash} \cong \Ai_{\htimes}^{\boxslash}$.
\end{enumerate}
\end{lemma}

\begin{proof}
We start with $(1)$. In order to define an inverse $\Ai^{\boxslash} \to \Em^{\boxslash}$, consider an object $(f, \theta) \in \Ai^{\boxslash}$, where $\theta$ is a lifting operation for squares $(a, b) \co i \to f$ with $i \in \Ai$. We need a way to canonically extend $\theta$ to all arrows in $\Em$. In order to do this, let us consider $j\co A \to B$ in $\Em$. Since $\catA$ is dense, we can express $B$ canonically as a colimit $B \cong \mathsf{colim}_{k \co \ocat{A} \to B} \ocat{A}$ with  $\ocat{A} \in \catA$. Moreover, in a topos, pullbacks commute with colimits, and thus we obtain
\[
j \cong \mathsf{colim}_{k \co \ocat{A} \to B} ( \Delta_k j) \rlap{,}
\]
where $\Delta_k j$ is the pullback of $j \co A \to B$ along $k \co \ocat{A} \to B$. Since $\Em$ is closed under base change, $\Delta_k j \in \Em$ and by definition, we get $ \Delta_k j \in \Ai$. A filler for a square $(a, b) \co j \to f$ is canonically obtained from the universal property of the colimit, applied to the collections of fillers given by $\theta$ relative to $\Delta_k j$ for each $k \co \ocat{A} \to B$.

For $(2)$, we proceed in a similar manner. Consider $(f, \theta) \in \Ai_{\htimes}^{\boxslash}$, we need to canonically extend $\theta$ to all arrows in $\Em_{\htimes}$. For this, consider $\delta^k \hat{\times} j \in \Em_{\htimes}$. By definition, $j \in \Em$ and thus $j \cong \mathsf{colim}_{k \co \ocat{A} \to B} (k^*j)$ by the previous argument. Since $\delta^k \htimes -$ is cocontinuous, we obtain:
\[
\delta^k \hat{\times} j \cong \delta^k \hat{\times} \mathsf{colim}_{k \co \ocat{A} \to B} (\Delta_k j) \cong \mathsf{colim}_{k \co \ocat{A} \to B} (\delta^k \htimes ( \Delta_k j)) \rlap{,}
\]
and by definition $\delta^k \htimes (\Delta_k ) \in \Ai_{\htimes}$. Once more, any square $(a, b) \co \delta^k \hat{\times} j \to f$ can be filled canonically by the universal property of the colimit applied to the collections of fillers given by $\theta$ relative to $\delta^k \htimes (\Delta_k j)$.
\end{proof}

\begin{proposition} \label{remark:strange}
Consider a suitable topos $(\catE, I, \Em)$. There exists two awfs~$(C, F_t)$ and~$(C_t, F)$ on~$\catE$ which are algebraically-free on  $\Em$ and on $\Em_{\htimes}$ respectively. 
\end{proposition}
\begin{proof}
Apply Garner's small object argument to $\Ai$ and $\Ai_{\htimes}$ respectively. Since $\catE$ is a Grothendieck topos, there exists a small dense subcategory $\catA$ of $\catE$ (for example, the full subcategory of compact objects for a large enough cardinal). By \cref{lemma:awfssmall}, the resulting awfs's are algebraically-free on $\Em$ and $\Em_{\htimes}$ respectively.
\end{proof}

%

\begin{remark} \label{remark:strange1}
By definition of algebraically-free awfs we have the following isomorfisms $\Ftalg \cong \category{TrivUniFib}$ and $\Falg \cong \category{UniFib}$. And, by \cref{prop:baf} we have back-and-forth functors $\Ftmaps \leftrightarrow \category{TrivUniFib}$ and $\Fmaps \leftrightarrow \category{UniFib}$.
\end{remark}

We proceed to show that, under some extra hypothesis, the awfs $(C, F_t)$ of uniform fibrations  is type-theoretic. We know that it has a functorial Frobenius structure by \cite[Theorem~8.8]{Gambino_2017} and
so we only need to construct a stable functorial choice of path objects on $(C_t, F)$. For this, we require the following construction. Given a topos $\catE$ equipped with an interval object $I$, there is a natural way to construct a stable and functorial factorisation of the diagonal: for a morphism~$f \co B \rightarrow A$, consider 
\[
\xymatrix{
B \ar[r]^{r_f} & Pf \ar[r]^(0.45){\rho_f} & B \times_A B \rlap{,}
}
\]
where the object $Pf$ and the map $r_f$ arise from the pullback diagram:
\begin{equation} 
\begin{gathered}
\xymatrix{
B  \ar@/^1.2pc/[drr]^{B^{\epsilon}} \ar@/_1.2pc/[ddr]_{f} \ar@{..>}[dr]^{r_f} &  & \\
    &  P f \drpullback \ar[r] \ar[d]   &  B^I \ar[d]^{f^I}\\
    &    A \ar[r]_{A^{\epsilon}}             &    A^I \rlap{.}
}
\end{gathered}
\label{Fig:WPO}
\end{equation}
Here, we use the abbreviation of $(-)^I$ for the exponential object $\hom(I, -)$ and denote by~$\epsilon \co I \rightarrow \bot$ the unique map to the terminal object. The second leg of the factorisation~$\rho_f \co P f \rightarrow B \times_A B$ is given by the universal property of $B \times_A B$ applied to the canonical source and target maps $s_f, t_f \co Pf \rightarrow B$ given by the composition of the arrow~$Pf \rightarrow B^I$ from the pullback square, and $B^{\delta^0}, B^{\delta^1} \co B^I \rightarrow B$ respectively. We denote the factorisation by $\lcat{P}_I$, so as to indicate that it was constructed from the interval $I$.

Let us provide an alternative construction of this factorisation which makes evident some intermediate steps and uses the adjunction $- \hat{\times} i \vdash \hat{\hom}(i, -)$ given by the pushout-product and pullback-exponential. Denote by $i \co \partial I \rightarrow I$ the boundary inclusion of the interval object and by $\iota^k \co \bot \rightarrow \partial I$ the composition of the boundary with the $k$-th endpoint inclusion. The following diagram expands the previous one, \emph{i.e.}~the exterior part is exactly the one in~\eqref{Fig:WPO}.
\begin{gather}
\begin{aligned}
\xymatrix{
B \ar@/^10pt/[rrrd]^{B^{\epsilon}} \ar[dr]^{r_f} \ar[ddr]|{\Delta_f} \ar@/_10pt/[dddr]|{1_B} \ar@/_30pt/[ddddr]_{f} &&&  \\
   & Pf \drpullback \ar[rr] \ar[d]^{\rho_f} && B^I \ar[d]_{\hat{\hom}(i, f)} \ar@/^35pt/[dd]^{\hat{\hom}(
  \delta^1, f)} \ar@/^90pt/[ddd]^{f^I} \\
   & B \times_A B \drpullback \ar[d]^{\pi_2} \ar[rr]|{\abra{\alpha_f, \lambda_f}} && A^I \times_{A^{\partial I}} B^{\partial I} \ar[d]_{1 \times_{A^{\iota^1}} B^{\iota^1}} \\
   & B \drpullback  \ar[rr]|{\abra{\beta_f, 1_B}} \ar[d]^f           && A^I \times_A B \ar[d]_{\pi_1} \\
   & A     \ar[rr]_{A^{\epsilon}}         && A^I \rlap{.}
}
\end{aligned}
\label{Fig:bigfatfig}
\end{gather}
The intermediate horizontal arrows $\lambda_f$, $\alpha_f$ and $\beta_f$ are given intuitively as follows. The map $\lambda_f$ sends a pair of points in $B \times_A B$ to the same pair of points but now in $B^{\partial I}$, $\alpha_f$ sends a similar pair of points $(b_1, b_2)$ to the reflexivity (constant) path on $f(b_1) = f(b_2)$, and map $\beta_f$ also sends a point $b$ to the reflexivity path on $f(b)$.

The next essential ingredient needed to prove that the factorisation $\lcat{P}_I = \abra{r, \rho}$ lifts to a stable functorial choice of path objects, is that of the categories of \emph{strong homotopy equivalences} and of \emph{strong deformation retracts}. We recall from \cite[Definition~4.1]{Gambino_2017} the definition of $k$-oriented strong homotopy equivalence (for $k \in \{ 0, 1 \}$) and from \cite[Lemma~8.1]{Gambino_2017} that they assemble into a category of arrows which we will call $\category{SE}_k$, we denote $\category{SE} \co= \category{SE}_0 + \category{SE}_1$. The definition of k-oriented strong deformation retracts is analogous; briefly, a k-oriented strong deformation retraction structure corresponds to a tuple $(g\co A \rightarrow B, r\co B \rightarrow A, h\co I \times B \rightarrow B)$ such that $rg = 1_A$, $h$ is an homotopy from $gr$ to $1_B$ or from $1_B$ to $gr$ respectively if $k$ is $0$ or $1$ and such that $h$ is degenerate in the image of $g$ (hence the strength). Strong deformation retracts assemble into categories of arrows $\category{SDR}_k$ depending on the orientation $k \in \{ 0, 1 \}$, and so we obtain $\category{SDR} \co= \category{SDR}_0 + \category{SDR}_1$ by taking their coproduct. Notice that any strong deformation retract is also a strong homotopy equivalence, i.e. there is a functor of category of arrows $\category{SDR} \rightarrow \category{SE}$ over the identity of the underlying category.

The following two lemmas constitute the first key results regarding the connection between the factorisation $\lcat{P}_I = \abra{r, \rho}$ and the awfs of uniform fibrations. 

\begin{lemma} \label{claim:1}
Consider a suitable topos $(\catE, I, \Em)$ and let $(C_t, F)$ be a corresponding awfs of uniform fibrations on $\catE$. Suppose that the following additional hypothesis holds:
\begin{enumerate}
\item[(M5)] Maps in $\Em$ are closed under pushout-product against the boundary inclusion $i \co \partial I \rightarrow I$, i.e. for any $j \in \Em$, we have that $i \hat{\times} j \in \Em$.
\end{enumerate}
Then, the second component of the factorisation $\lcat{P}_I$, i.e $\rho \co \catE^{\rightarrow} \rightarrow \catE^{\rightarrow}$, lifts to a functor $\rho \co \Fmaps \rightarrow \Fmaps$.
\end{lemma}

\begin{proof} Since the awfs of trivial uniform fibrations $(C, F_t)$ is suitable (see \cite[Definition~7.1]{Gambino_2017}),  the functor $\delta^k \hat{\times} - $ lifts to the category $\Cmaps$ and $\delta^k \hat{\times} - $ also factors though the category \category{SE} of strong homotopy equivalences
by \cite[Lemma~8.4]{Gambino_2017}. Combining these two facts, we obtain a lift  
\[
\delta^k \hat{\times} (-) \co \Cmaps \to  \Cmaps \times_{\catE^{\rightarrow}} \category{SE} \rlap{.}
\]
By \cite[Proposition~8.5]{Gambino_2017}, we have a functor $\Cmaps \times_{\catE^{\rightarrow}} \category{SE} \rightarrow \Ctmaps$ over $\catE^{\rightarrow}$, composing with the one above, we obtain a lift of $\delta^k \hat{\times} - $
\begin{equation}
\label{equ:first-lift}
\delta^k \hat{\times} (-)  \co  \Cmaps \to \Ctmaps \rlap{.}
\end{equation} 
By  functorial orthogonality arguments with respect to the pushout-product and pullback-exponential constructions (see \cite[Proposition~5.11]{Gambino_2017} and \cite[Remark~5.12]{Gambino_2017}) together with the hypothesis (M5), the functor $i \hat{\times} - \co \Em \rightarrow \Em$ lifts to the category $\Cmaps$, 
\begin{equation}
\label{equ:second-lift} 
i \hat{\times} (-) \co  \Cmaps \to \Cmaps  \rlap{.}
\end{equation}

Applying the lifts in~\eqref{equ:first-lift} and~\eqref{equ:second-lift}, together with the fact that $(C, F_t)$ is algebraically-free on the category of arrows $\Em \rightarrow \catE^{\rightarrow}$, as witnessed by the functor $\eta \co \Em \rightarrow \Calg$, we obtain the diagram
\[
\xymatrix{
\Em \ar[r]^-{\tilde{\eta}} \ar[dr] & \Cmaps \ar[r]^{i \hat{\times} - } \ar[d] & \Cmaps \ar[r]^{\delta^k \hat{\times} -} \ar[d]& \Ctmaps \ar[d] \\
&  \catE^{\rightarrow} \ar[r]_{i \hat{\times} -} &  \catE^{\rightarrow} \ar[r]_{\delta^k \hat{\times} -}  & \catE^{\rightarrow} \rlap{,}
}
\] 
where $\tilde{\eta}$ is the composite of  $\eta$ and the forgetful functor from $C$-algebras to $C$-maps. 

By symmetry of the pushout-product functor, we obtain a natural isomorphism between $i \htimes \delta^k \htimes -$ and $\delta^k \htimes i \htimes -$.
We can transfer the algebraic structure along this natural isomorphism in order to obtain the following lift:
\[
\xymatrix{
\Em \ar[d] \ar[rr] && \Ctmaps \ar[d] \\
\catE^{\rightarrow} \ar[r]_{\delta_k \htimes -} & \catE^{\rightarrow} \ar[r]_{i \htimes -} & \catE^{\rightarrow}
\rlap{.}
}
\]
Taking the coproduct of these lifts for $k = 0, 1$ we obtain a lift of $i \htimes -$, \rlap{.}
\begin{equation}
\label{equ:third-lift}
{i \htimes (-)} \co  \Em_{\htimes} \to \Ctmaps \rlap{.}
\end{equation}

Using that $\Ctmaps \cong {^\boxslash}\Falg$ (cf.~\cref{prop:liftings}) and that $(C_t, F)$ is algebraically-free on $\Em_{\htimes}$, we can apply \cite[Proposition~5.9]{Gambino_2017} to~\eqref{equ:third-lift} and obtain
\[
\xymatrix{
\Falg \ar[rr]^{\hhom(i, -)} \ar[d] && (\Em_{\htimes})^{\boxslash} \ar[d] \ar[r]^{\cong} & \Falg \ar[dl]  \\
\catE^{\rightarrow} \ar[rr]_{\hhom(i, -)} && \catE^{\rightarrow}  \rlap{.} &
}
\]

By the top pullback square in \eqref{Fig:bigfatfig}, the morphism $\rho_f \co Pf \rightarrow B \times_A B$ is obtained in the following two steps:
\[
f \mapsto \hhom(i, f) \mapsto \abra{\alpha_f, \lambda_f}^*\hhom(i, f) = \rho_f \rlap{,}
\]
i.e. by first applying $\hhom(i, -)$ and then pulling back along $\abra{\alpha_f, \lambda_f}$. Since we have lifts of $\hhom(i, -)$ and of the pullback functor to the category of $F$-algebras, we obtain a lift of $\rho$, 
\[
\xymatrix{
\Falg \ar@{.>}@/^20pt/[rrrr]^{\rho} \ar[rr]_{\hhom(i, -)} && \Falg \ar[rr]_{PB(-, \abra{\alpha, \lambda})} && \Falg \rlap{.}
}
\]

Finally, as we are working with an algebraically-free awfs, we have lifts back-and-forth between $\Ralg$ and $\Rmaps$ over $\catE^{\rightarrow}$ (cf.~\cref{prop:baf}), and thus we can transfer the lift of $\rho$ from the category of $R$-algebras to that of $R$-maps. 
\end{proof}

\begin{lemma} \label{claim:2}
Consider a suitable topos $(\catE, I, \Em)$ and let $(C_t, F)$ be a corresponding awfs of uniform fibrations on $\catE$. Then the first component $r \co  \catE^{\rightarrow} \rightarrow \catE^{\rightarrow}$ of the factorisation~$\lcat{P}_I$ lifts to the category of strong deformation retracts, 
\[
\xymatrix{
\Fmaps \ar[d] \ar@{.>}[r]^{{r}} & \category{SDR} \ar[d] \\
\catE^{\rightarrow} \ar[r]_r & \catE^{\rightarrow} \rlap{.}
}
\]
\end{lemma}
\begin{proof}
We first show that the target map functor (that takes a map $f \co B \rightarrow A$ to a map $t_{f} \co Pf \rightarrow B$) lifts to a functor from $\Fmaps$ to $\Ftmaps$. Using that we have a lift $\delta^1 \htimes - \co \Cmaps \rightarrow \Ctmaps$ as shown in the proof of \cref{claim:1}, we can transpose using \cite[Proposition~5.9]{Gambino_2017} to obtain a lift of $\hhom(\delta^1, -)$:
\[
\hhom(\delta^1, -) \co \Falg \to \Ftalg \rlap{.}
\]

Looking at \eqref{Fig:bigfatfig}, note that $t_f \co Pf \rightarrow B$ is obtained by applying $\hhom(\delta^1, -)$ to $f$ and then pulling back along $\abra{\beta_f, 1_B}$. Thus, the functor mapping $f \mapsto t_f$ lifts as
\begin{gather}
\begin{aligned}
\xymatrix{
\Falg \ar@{.>}@/^20pt/[rrrr]^{t_{(-)}} \ar[rr]_{\hhom(\delta^1, -)}  && \Ftalg \ar[rr]_{PB(-, \abra{\beta, 1})} && \Ftalg \rlap{.}
}
\end{aligned}
\label{Fig:liftdelta1}
\end{gather}
Since both awfs in question are algebraically-free, we can apply \cref{prop:baf} to obtain the desired lift.

Let us return to the task of finding a lift of the functor $r \co \catE^{\rightarrow} \rightarrow \catE^{\rightarrow}$ to a functor $r \co \Fmaps \rightarrow \category{SDR}$. For this, we show that for each uniform fibration $(f, s) \co B \rightarrow A$ the target map $t_f \co Pf \rightarrow B$ is a strong homotopy retraction of $r_f \co B \rightarrow Pf$.

Looking again at~\eqref{Fig:bigfatfig} it is clear that $t_f \circ r_f = 1_B$. Thus, we are left with the task of constructing an homotopy $H \co r_f \circ t_f \sim 1_{P f}$, for this consider the  commutative diagram
\[
\xymatrix{
P f \ar[rr]^{\abra{r_f \circ t_f,  1_{P f}}} \ar[d]_{B^{\epsilon} \circ t_f} && P f^{\partial I} \ar[d]^{t_f^{\partial I}} \\
B^I \ar[rr]_{B^i} && B^{\partial I}  \rlap{,}
}
\]
where the top horizontal arrow is given by the universal property of~$P f^{\partial I} \cong P f \times P f$. This gives us an arrow into the pullback
\[
\tilde{H} \co P f \rightarrow B^I \times_{B^{\partial I}} P f^{\partial I}.
\]

We already have a lift of the target map $t_{(-)} \co \Fmaps \rightarrow \Ftmaps$. Combining this with the fact that $\hhom(i, -)$ lifts to $\Ftmaps$ (which follows by similar arguments to those used in the proof of \cref{claim:1}), $\hhom(i, t_{(-)})$ lifts to a functor
\[
\hhom(i, t_{(-)}) \co  \Fmaps \to \Ftmaps \rlap{,}
\]
which we can apply to $f$ to obtain a uniform trivial fibration $\hhom(i, t_f)$.

By part~(M2) of \cref{defn:suittops}, for every object $X \in \catE$, the map $\emptyset \rightarrow X$ is in $\Em$. Using this, we obtain a morphism $H$ as the canonical filler in 
\[
\xymatrix{
\emptyset \ar[d] \ar[rr] && P f^{I} \ar[d]^{\hhom(i, t_f)} \\
P f \ar[rr] _{\tilde{H}} \ar@{.>}[urr]|H && B^I \times_{B^{\partial I}} P f^{\partial I} \rlap{.}
}
\]
It is straightforward to verify that this $H$ is actually an homotopy from $r_f \circ t_f$ to $1_{P f}$. This shows that $t_f$ is a strong deformation retract of $r_f$.

We have given the action of the desired lift $r \co \Fmaps \rightarrow \category{SDR} $ on objects . To show that this construction is functorial on $f$, consider a morphism of $\Fmaps$ $(h, k) \co f' \rightarrow f$. Since the factorisation of the diagonal is functorial, we obtain the  diagram
\[
\xymatrix{
B' \ar[rr]^{h} \ar[d]_{r_{f'}} && B \ar[d]^{r_f} \\
P f' \ar[rr]|{P(h, k)} \ar[d]_{t_{f'}} && P f \ar[d]^{t_f} \\
B' \ar[rr]_{h} && B \rlap{.}
}
\]
The bottom square is a morphism of $\Ftmaps$ since it is the result of applying the lift of~$t_{(-)}$ of (\ref{Fig:liftdelta1}) to the square~$(h, k)$. 
Let us prove that $(h, P(h, k)) \co r_{f'} \rightarrow r_{f}$ is a morphism of strong deformation retracts. Looking at the definition of a morphism of homotopy equivalences (in the paragraph before \cite[Lemma~8.1]{Gambino_2017}), we observe that the only thing we need to show is that the following diagram commutes:
\[
\xymatrix{
P f' \ar[rr]^{P(h, k)} \ar[d]_{H'} && P f \ar[d]^{H} \\
P f'^I \ar[rr]_{P(h, k)^I} && P f^I \rlap{,}
}
\]
where the left and right horizontal arrows are the homotopies witnessing that $r_{f'}$ and $r_f$ respectively are strong deformation retracts. For this, we make use of the naturality of the filling operations. Consider the diagrams:
\begin{equation}
\label{equ:top-diag}
\begin{gathered} 
\xymatrix{
\emptyset \ar[d] \ar[rr] && \emptyset \ar[d] \ar[rr] && P f^I \ar[d]^{\hhom(i, t_f)} \\
P f' \ar[rr]_{P (h, k)} \ar@{.>}[urrrr]|(.3){L'} && P f \ar@{.>}[urr]|H \ar[rr]_{\tilde{H}} && f^I \times_{B^{\partial I}} P f^I \rlap{,} 
} \end{gathered} 
\end{equation}
\begin{equation}
\label{equ:bottom-diag}
\begin{gathered} 
\xymatrix{
\emptyset \ar[d] \ar[rr] && P f'^I \ar[d]|{\hhom(i, t_{f'})} \ar[rr]^{P (h, k)^I} && P f^I \ar[d]^{\hhom(i, t_f)} \\
P f' \ar[rr]_{\tilde{H}'} \ar@{.>}[urrrr]|(.7)L \ar@{.>}[urr]|{H'} && f'^I \times_{B'^{\partial I}} P f'^I \ar[rr]_{h^I \times_{h^{\partial I}} P(h, k)^{\partial I}} && f^I \times_{B^{\partial I}} P f^I \rlap{.}
}
 \end{gathered}
\end{equation} 
The left square in~\eqref{equ:top-diag} is a morphism in $\Em$ since it is  Cartesian. The right square of~\eqref{equ:bottom-diag} is a morphism of $F_t$-maps since it is the result of applying the lift~$\hhom(i, t_{(-)}) \co \Fmaps \rightarrow \Ftmaps$ to the square $(h, k)$ which is, by hypothesis, a morphism of $F$-maps. Hence, the corresponding lifts cohere.  

Since the construction of the maps $\tilde{H}$ and $\tilde{H}'$ is given by a universal property, it
is functorial and so
the diagram \[
\xymatrix{
P f' \ar[rrr]^{P (h, k)} \ar[d]_{\tilde{H}'} &&& P f \ar[d]^{\tilde{H}} \\
 B'^I \times_{B'^{\partial I}} P f'^I \ar[rrr]_{h^I \times_{h^{\partial I}} P(h, k)^{\partial I}} &&& B^I \times_{B^{\partial I}} P f^I
}
\] 
commutes. Thus, the composition of the bottom horizontal arrows in~\eqref{equ:top-diag} 
and~\eqref{equ:bottom-diag} coincide. This makes the filler $L'$ and $L$ in diagrams \ref{equ:top-diag} and \ref{equ:bottom-diag} respectively, the same morphism and thus 
$H \circ P (h, k) = L' = L = P (h , k)^I \circ H'$, as required.
\end{proof}

Using \cref{claim:1} and \cref{claim:2}, we can prove the following proposition. 

\begin{proposition} \label{theorem:main1}
Consider a suitable topos $(\catE, I, \Em)$ and let $(C_t, F)$ be a corresponding awfs of uniform fibrations on $\catE$. Suppose that the following additional hypotheses hold:
\begin{enumerate}
\item[(M5)] Maps in $\Em$ are closed under pushout-product against the boundary inclusion $i \co \partial I \rightarrow I$, i.e. for any $j \in \Em$, we have that $i \hat{\times} j \in \Em$.
\item[(M6)] For any $f \co B \rightarrow A$ in $\catE$, the first leg map $r_f \co B \rightarrow Pf$ from the factorisation of the diagonal $\lcat{P}_I$ belongs to $\Em$:
\end{enumerate}
Then, the factorisation of diagonal $\lcat{P}_I = \abra{r, \rho}$ induced from the interval object, lifts to a stable functorial choice of path objects for $(C_t, F)$:
\[
\lcat{P}_I \co 
\Fmaps \to \Ctmaps \times_{\catE} \Fmaps \rlap{.}
\] 
\end{proposition}

\begin{proof}
Recall that the factorisation of the diagonal $\lcat{P}_I = \abra{r, \rho}$ is divided into two functors~$r, \rho \co \catE^{\rightarrow} \rightarrow \catE^{\rightarrow}$. By \cref{claim:1}, $\rho$ lifts to the category $\Fmaps$,
so it remains to show that $r \co \catE^{\rightarrow} \rightarrow \catE^{\rightarrow}$ lifts to a functor $r \co \Fmaps \rightarrow \Ctmaps$. This follows from two observations. First, since the factorisation of the diagonal is stable,~$r$ preserves Cartesian squares and thus, by item (M6) in the hypothesis of the theorem,~$r$ lifts 
to~$\Em$ (considered as a category of arrows), $r \co \catE^{\rightarrow} \to \Em$.
Secondly, consider the unit $\eta_{\Em} \co \Em \rightarrow {^\boxslash}(\Em^{\boxslash})$ of the orthogonality adjunction of~\ref{equ:orth-adj} and note that since $(C, F_t)$ is algebraically-free on $\Em$, we obtain a morphism in the slice over $\catE^{\rightarrow}$, 
\[
\eta_{\Em} \co  \Em \to \Cmaps \rlap{.}
\]
We can compose these last two lifts to obtain $r \co \catE^{\rightarrow} \rightarrow \Cmaps$.

Finally, we can combine the lifts $r \co \Fmaps \rightarrow \category{SDR}$ from \cref{claim:2} with $r \co \catE^{\rightarrow} \rightarrow \Cmaps$ and apply \cite[Proposition~8.5]{Gambino_2017} in order to obtain the desired lift of $r$, 
\[
\xymatrix{
\Fmaps \ar\ar@{.>}@/^20pt/[rr]^r \ar[r] & \Cmaps \times_{\catE^{\rightarrow}} \category{SE} \ar[r] & \Ctmaps
} \rlap{.} \qedhere
\]
\end{proof}

We introduce the following definition to summarise our results. 

\begin{definition} \label{defn:ttst}
A {\em type-theoretic suitable topos} consists of a suitable topos $(\catE, I, \Em)$ (see \cref{defn:suittops}) which moreover satisfy the conditions (M5) and (M6) in the hypothesis of \cref{theorem:main1}. 
\end{definition} 


\begin{theorem} \label{theorem:main}
Let $(\catE, I, \Em)$ be a type-theoretic suitable topos, and let $(C_t, F)$ be the awfs of uniform fibrations on $\catE$. Then $(C_t, F)$ is equipped with the structure of a type-theoretic awfs.
\end{theorem}

\begin{proof}
The result follows from \cite[Theorem~8.8]{Gambino_2017} and \cref{theorem:main1}.
\end{proof}

The next result allows us to construct examples of type-theoretic suitable topos.

\begin{proposition} \label{prop:main1}
Consider $\catE$ be a Grothendieck topos equipped with an interval object $I$ with connections. Let $\Em_{all}$ be the class that consists of all monomophisms of $\catE$. Then the tuple $(\catE, I, \Em_{all})$ is a type-theoretic suitable topos.
\end{proposition}

\begin{proof}
We need to verify conditions (M1)-(M6) from the definition of type-theoretic suitable topos. By elementary properties of monomorphisms, it is clear that (M1)-(M3) hold. Condition (M4) follows because in the topos $\catE$ the pushout-product construction~$\delta^k \htimes j$ for a given monomorphism~$j \co A \rightarrow B$, computes the join (or union) of the subobjects~$\delta^k \times B \co B \rightarrow I \times B$ and $I \times j \co I \times A \rightarrow I \times B$, which is again a subobject of $I \times B$ and in particular a monomorphism. The same arguments applies for condition (M5). Finally condition (M6) follows since for any map $f \co X \rightarrow Y$, the morphism $r_f \co X \rightarrow Pf$ is the section of the target map $t_f \co Pf \rightarrow X$ and in particular, it is a monomorphism. 
\end{proof}

\begin{example} We can instantiate \cref{prop:main1} on the presheaf toposes of simplicial sets $\SSet$ and of cubical sets $\CSet$ equipped with the obvious choices of interval objects given by the representable $1$-simplex and $1$-cube respectively. We thus obtain type-theoretic awfs on $\SSet$ and $\CSet$. 
In $\SSet$, assuming classical logic, a morphism admits the structure of a uniform fibration if and only if it is a Kan fibration by~\cite[Theorem~9.9]{Gambino_2017}.
\end{example}

\begin{remark} \label{note:nonconstruct}
Although the proof of \cref{theorem:main} is constructive, in order to construct a univalent universe \`{a} la Hofmann-Streicher \cite{Hofmann:aa} in a constructive setting, 
it is necessary to restrict the category $\Em_{all}$ of generating monomorphisms to that of decidable ones; \emph{i.e.}~those mononomorphism  that have level-wise decidable image\cite{Orton_2017}.
The arguments in this section do not apply if we take $\Em_{dec}$ as the category of generating monomorphisms, where $\Em_{dec}$ is the subclass of $\Em_{all}$ of decidable monomorphisms (for either $\SSet$ or $\CSet$). The issue lies in condition (M6), 
\emph{i.e.}~that the first leg  $r_{f} \co X \rightarrow Pf$  of the factorisation of the diagonal of a morphism $f \co X \rightarrow Y$ is in $\Em_{dec}$. Intuitively, the morphism $r_f$ maps an object of $x$ of $X$ to the degenerate path on $x$, this morphism is not decidable because, in general, it is not possible to decide degeneracies \cite{Bezem_2015}. Because of this, in the next section we turn our attention to normal uniform fibrations (see~\cref{rem:follow-up-on-costructivity}). An alternative approach would be to restrict attention to cofibrant objects, as in~\cite{Gambino_2019}.
\end{remark}

\section{Normal uniform fibrations} \label{sect:NUF}
In this section, we develop the notion of a normal uniform fibration in the context of a suitable topos $(\catE, I, \Em)$ (\cref{defn:suittops}). Recall from the discussion before \cref{remark:strange} that the category of arrows of uniform fibrations was constructed by right orthogonality form the category of arrows $\Em_{\htimes}$ over $\catE$, whose objects are maps of the form $\delta^k \htimes j$ with $j \in \Em$ and $k \in \{0,1\}$.  Moreover, recall from  \cref{remark:strange} that  we cannot use Garner's small object argument directly with $\Em_{\htimes}$ to construct the awfs of uniform fibrations $(C_t, F)$ as $\Em_{\htimes}$ is not small; instead, we need to restrict to the small category of arrows $\Ai_{\htimes}$ that consists of arrows $\delta^k \htimes j$ with $j \in \Ai$, where $\Ai$ consists of the arrows in $\Em$ whose codomain lie in a fixed small dense subcategory of $\catE$.  

We define a new category of arrows $\nufib \rightarrow \catE^{\rightarrow}$ such that a right $\nufib$-map will consist of a uniform fibration with an extra normality property. The idea is that $\nufib \rightarrow \catE^{\rightarrow}$ is obtained form $\ufib \rightarrow \catE^{\rightarrow}$ by adding, for each generating monomorphism $i : A \mono B$ and for $k \in \{0,1\}$, the coherence square on the left in
\begin{equation}
\label{Fig:fignotbig}
\begin{gathered} 
\xymatrix{
B +_{A} (I \times A) \ar[d]_{\delta^k \htimes i} \ar[r]^(0.70){\squi_k(i)} & B \ar@{=}[d] \\
I \times B \ar[r]_(0.60){\epsilon \times B} & B \rlap{,} }
\end{gathered}
\qquad
\begin{gathered} 
\xymatrix{
A \ar[r]^i \ar[d]_{\delta^k \times A} & B \ar[d]^{\delta^k \times B} \ar@/^18pt/[ddr]^{1_B} &\\
I \times A \ar[r] \ar@/_8pt/[rd]_{\epsilon \times A} & B +_{A} (I \times A) \poexcursion \ar@{.>}[dr]|{\squi_k(i)} & \\ 
&A\ar[r]_i & B \, ,
}
\end{gathered}
\end{equation}
where $\squi_k(i) \co B +_{A} (I \times A) \rightarrow B$ is the universal map out of the pushout as described on the right of the previous diagram. The arrows $\epsilon \times B$ and $\epsilon \times A$ are the projections from the second component of the product. We refer to the square on the left of (\ref{Fig:fignotbig}) as the \emph{$k$-squash square} of $i \co A \mono B$ and we denote it by \[\squish_k(i) \co \delta^k \htimes i \rightarrow 1_B.\] The name follows the intuition of squashing the mapping cylinder in the direction of the interval (\emph{i.e.}~the filling direction). The following technical result about squash squares will be needed in what follows.

\begin{lemma} \label{lem:squish}
Let $k \in \{0,1\}$ and consider monomophisms $i \co A \mono B$ and $j \co C \mono D$. Then applying the pushout-product functor $(j \htimes -) \co \catE^{\rightarrow} \rightarrow \catE^{\rightarrow}$ to the k-squash square of $i \co A \mono B$, produces the k-squash square of $j \htimes i$; that is: 
\[
j \htimes (\squish_k(i)) \cong \squish_k(j \htimes i) \co \delta^k \htimes (j \htimes i) \rightarrow 1_{D \times B} \, .
\]
\end{lemma}

\begin{proof}
If we apply $(j \htimes -) \co \catE^{\rightarrow} \rightarrow \catE^{\rightarrow}$ to the k-squash square of $i \co A \mono B$, using that the pushout-product is symmetric and associative, we get the following square:
\[
\xymatrix{
\dom(\delta^k \htimes (j \htimes i)) \ar[d]_{\delta^k \htimes (j \htimes i)} \ar[rr]^{\Theta} && D \times B \ar@{=}[d] \\
I \times (D \times B) \ar[rr]_{\epsilon \times (D \times B)} && D \times B
}
\]
where we only need to verify that the top horizontal arrow $\Theta$ is the squash morphism, that is, we need to verify that $\Theta = \squi_k(j \htimes i) \co \dom(\delta^k \htimes (j \htimes i)) \rightarrow D \times B$, but this follows since the diagram commutes.
\end{proof}

We now proceed to construct the arrow category $\nufib \rightarrow \catE^{\rightarrow}$ that will generate the category of normal uniform fibrations. We do this a follows. First let us denote by $\icat{I}$ the `walking arrow', that is the poset with two objects $0 < 1$ considered as a category, this has the structure of an interval object in $\scat{Cat}$ and we denote the inclusions by: \[\xymatrix{{*} \ar@<-.5ex>[r]_{\iota^1} \ar@<.5ex>[r]^{\iota^0} & \icat{I}}\]

We define $\nufib := \icat{I} \times \ufib$, where $\ufib$ is the generating category of uniform fibrations. The functor down to $\catE^{\rightarrow}$ is determined by the following two properties.
\begin{enumerate}
\item The following diagram commutes:
\[
\xymatrix{
\ufib \ar[dr]_{\uufib} & \nufib \ar[l]_{\rho_0} \ar[d]^{} \ar[r]^{\rho_1}& \ufib \ar[dl]^{\epsilon_{cod}}\\
&\catE^{\rightarrow}&
}
\]
where the map $\epsilon_{cod} \co \ufib \rightarrow \catE^{\rightarrow}$ sends an object $i \in \ufib$ to the identity arrow on the codomain of $i$. 

\item \label{item:2} For $k\in\{0,1\}$ and for each $i \co A \mono \ocat{A}$ in $\lcat{I}$, the functor $\unufib$ takes the arrow in $\icat{I} \times \ufib$ of the form $\icat{I} \times i \co \{0\} \times i \rightarrow \{1\} \times i$, to the $k$-squash square of $i$; \emph{i.e.}~$\squish_k(i) \co \delta^k \htimes i \rightarrow 1_{\ocat{A}}$.
\end{enumerate}
In other words, $\nufib \rightarrow \catE^{\rightarrow}$ is a natural transformation:
$
\uufib \rightarrow \epsilon_{cod} \co \ufib \rightarrow \catE^{\rightarrow}
$
whose components are the $k$-squash squares. 

We define $\category{NrmUniFib} \rightarrow \catE^{\rightarrow}$ to be the category of arrows of right $\nufib$-maps in $\catE$, and we call its objects {\em normal uniform fibrations}. Using Garner's small object argument \cite[Theorem~4.4]{Garner_2009} along the lines of \cref{lemma:awfssmall} and \cref{remark:strange}, we obtain the following result.

\begin{theorem} \label{thm:NUF}
There is an algebraically-free awfs on the category of arrows $\nufib \rightarrow \catE^{\rightarrow}$, denoted by $(NC_t, NF)$, whose category of $NF$-algebras is that of normal uniform fibrations. \qed
\end{theorem} 

Let us observe that the forgetful functor into $\catE^{\rightarrow}$ factors through the category of uniform fibrations, i.e. we have a commutative diagram:
\[
\xymatrix{
\category{NrmUniFib} \ar[rr]^U \ar[dr] && \category{UniFib} \ar[dl] \\
&\catE^{\rightarrow}\, .&
}
\] 
Moreover, we can prove the following lemma.
\begin{lemma} \label{lem:Uff}
The forgetful functor $U\co \category{NrmUniFib} \rightarrow \category{UniFib}$ is fully faithful.
\end{lemma}
\begin{proof}
This follows intuitively by noticing that the structure of a normal uniform fibration does not add any new lifting problems to that of a uniform fibrations; this is because the only new vertical arrows we are adding are identities and every morphism has a unique lift against them. Concretely, if $(f, \phi) \in \category{NrmUniFib}$ and if $(f, \theta) \in \category{UniFib}$, then both lifting structures $\phi$ and $\theta$ produce lifts against the exactly the same squares, the difference is that $\phi$ may have additional coherence properties. 
\end{proof}

In the following proposition we characterise those uniform fibration structures that are normal. We use the following terminology: we say that a morphism $\theta \co I \times B \rightarrow X$ is \emph{degenerate in the lifting direction} if it factors through the projection $\rho_1 \co I \times B \rightarrow B$ via some arrow $\theta^* \co B \rightarrow X$; we call $\theta^*$ the \emph{lifting degeneracy} of $b$.

\begin{proposition} \label{prop:nufib}
Let $(f, \theta) \in \category{UniFib}$ then the following are equivalent:
\begin{enumerate}
\item $(f, \theta)$ is a normal uniform fibration.
\item For any generating monomorphism $i \co A \mono \ocat{A}$ in $\lcat{I}$ and for any square:
\[
\xymatrix{
\ocat{A} +_A (I \times A) \ar[rr]^{a} \ar[d]_{\delta^{k} \htimes i} && X \ar[d]^{f} \\
I \times \ocat{A} \ar@{.>}[urr]|{\theta_{i}(a,b)} \ar[rr]_b &&Y \rlap{,}
}
\]
if the square factors through the squash square of $i$ as $\xymatrix{\delta^k \htimes i \ar[r]^(0.55){\squish_k(i)} & 1_{\ocat{A}} \ar[r]^{(a^* , b^*)} & f}$, then the lift $\theta_{i}(a,b)$ is degenerate in the lifting direction with $a^*$ as lifting degeneracy.
\item For any generating monomorphism $i \co A \mono B$ in $\Em$ and for any square:
\[
\xymatrix{
B +_{A} (I \times A) \ar[rr]^{a} \ar[d]_{\delta^{k} \htimes i} && X \ar[d]^{f} \\
I \times B \ar@{.>}[urr]|{\theta_{i}(a,b)} \ar[rr]_b &&Y
}
\]
if the square factors through the squash square of $i$ as $\xymatrix{\delta^k \htimes i \ar[r]^(0.55){\squish_k(i)} & 1_{B} \ar[r]^{(a^* , b^*)} & f}$, then the lift $\theta_{i}(a,b)$ is degenerate in the lifting direction with $a^*$ as lifting degeneracy.
\end{enumerate}
\end{proposition}

\begin{proof}
Let us first assume that $(f, \theta)$ is a normal uniform fibration. It is easy to see that item (2) holds, for this consider the diagram:
\[
\xymatrix{
\ocat{A} +_A (I \times A) \ar[d]_{\delta^{k} \htimes i}  \ar[r]^(0.67){\squi_k(i)}  & \ocat{A} \ar@{=}[d] \ar[rr]^{a^*}  && X \ar[d]^{f} \\
I \times \ocat{A} \ar@{.>}[urrr]|(0.25){\theta} \ar[r]_{\rho_1} & \ocat{A} \ar@{.>}[urr]|{a^*} \ar[rr]_{b*} && Y
}
\]
it is clear that the lifts cohere because the left square is by definition a morphism in (the image of) $\nufib \rightarrow \catE^{\rightarrow}$.

It is also easy to see that (2) implies (1), this follows since the uniform fibration structure $\theta$ already provides lifts against all lifting problems coming form $\nufib$, moreover, the lifts will also cohere with all the squares coming from $\uufib \co \ufib \rightarrow \catE^{\rightarrow}$. So we only need to verify that it coheres with the squash squares, but these squares are precisely those as in the hypothesis of item (2). 

It is clear that (3) implies (2). For the converse let us first observe, using that colimits in 
$\catE$ are universal, that any monomorphism $i \co A \mono B$, is the colimit over the generalised elements with domain on the dense subcategory $\catA$ used to define the category of arrows~$\Ai$ (see \cref{remark:strange}); that is, 
\[
i \cong \colim_{\substack{x\co \ocat{A} \rightarrow B \\ \ocat{A} \in \catA}} x^*(i) \rlap{,}
\] 
where for each $x \co \ocat{A} \rightarrow B$ we denote by $x^*(i)$ the pullback of $i$ along $x$. Now, since $\delta^k \htimes - \co \catE^{\rightarrow} \rightarrow \catE^{\rightarrow}$ is cocontinuous, we have that:
\[
\colim_{\substack{x\co \ocat{A} \rightarrow B \\ \ocat{A} \in \catA}} (\delta^k \htimes (x^*(i))) \cong \delta^k \htimes \colim_{\substack{x\co \ocat{A} \rightarrow B \\ \ocat{A} \in \catA}} x^*(i)  \cong \delta^k \htimes i
\]  

Let us suppose that (2) holds, and that we have a diagram as in item (3). Then for each generalised element $x \co \ocat{A} \rightarrow B$ with $\ocat{A} \in \catA$, we have a square:
\[
\xymatrix{
\ocat{A} +_{x^*(A)} (I \times x^*(A)) \ar[d]_{\delta^k \htimes x^*(i)} \ar[r]^(0.60){\iota_x} & B +_{A} (I \times A) \ar[rr]^{a} \ar[d]|(0.30){\delta^{k} \htimes i} && X \ar[d]^{f} \\
I \times \ocat{A} \ar[r]_{I \times x} \ar@{.>}[urrr]|(0.25){\theta_{x*(i)}} &I \times B \ar@{.>}[urr]|{\theta_i} \ar[rr]_b &&Y
}
\]
where the left square is the colimit inclusion corresponding to $x \co \ocat{A} \rightarrow B$. The commutation of the respective triangle is obtained by the universal property of the colimit. 

Finally, if the square on the right factors through a squash square 
\[
\xymatrix{\delta^k \htimes i \ar[rr]^(0.55){\squish_k(i)} && 1_{B} \ar[rr]^{(a^* , b^*)} && f}
\]
then (by naturality) the outer square also factor through a squash square and thus the lift $\theta_{x^*(i)}$ is degenerate with $a^* \iota_x $ as lifting degeneracy. This implies by the uniqueness of the universal property, that also $\theta_i$ is degenerate with $a^*$ as lifting degeneracy. 
\end{proof}

\begin{remark} \label{rem:nufiso}
To guide our intuition towards normal uniform fibrations, we can compare the notions of normality for cloven isofibrations in groupoids and for uniform fibrations in simplicial sets. For this, we consider the awfs of (normal) uniform fibrations on simplicial sets constructed from the suitable topos structure consisting of the $1$-simplex as the interval object and the class $\Em_{all}$ of all monomorphisms as the class of generating monomorphisms.
It is not hard to show that the following are pullback squares:
\[
\xymatrix{
\category{NrmFib} \ar[r]^{\tilde{N}} \ar[d] \pbcorner & \category{NrmUniFib} \ar[d] \\
\category{ClFib} \ar[r]^{\tilde{N}} \ar[d] \pbcorner & \category{UniFib} \ar[d] \\
\scat{Grd}^{\rightarrow} \ar[r]_{N} & \sSet^{\rightarrow} \rlap{.}
}
\]
Here, the categories $\category{ClFib}$ and $\category{NrmFib}$ are those of cloven isofibrations and normal cloven isofibrations in groupoids while the horizontal arrows are given by the nerve functor and its respective lifts.
This shows how the notion of uniform fibration (respectively normal uniform fibration) is a generalisation to higher dimensions of the notion of cloven isofibration (respectively normal cloven isofirations). 
\end{remark}

The category of arrows of \emph{normal trivial cofibrations} is defined to be the category of $NC_t$-maps with respect to the awfs of normal uniform fibrations \cref{thm:NUF}. Alternatively, it is the left orthogonal category of arrows of $\category{NrmUniFib}$. We will denote it by $\category{NrmTrivCof}$.
Even if we do not know a complete characterisation of normal trivial cofibrations, we have is a general method for constructing normal trivial cofibrations from a structure that is easier to handle. For this, let us recall the categories $\category{SDR}$ of strong deformation retractions and $\category{SE}$ of strong homotopy equivalences from \cref{sect:UF} defined immediately before \cref{claim:1}. In the next proposition, we observe that every strong deformation retract has (uniformly) the structure of a normal trivial cofibration. Normality is an essential ingredient in the proof, in particular, a similar result would not hold for uniform fibrations. 

\begin{proposition} \label{thm:SDRNTC}
There is a functor from the category strong deformation retracts $\category{SDR}$ to that of normal trivial cofibrations \category{NrmTrivCof},
\[
\xymatrix{
\category{SDR} \ar[rr]^{\Psi} && \category{NrmTrivCof} \rlap{.}
}
\]
\end{proposition}

\begin{proof}
 Let $(g, r, h) \in \category{SDR}$ which we assume to be $0$-oriented (the other case being analogous). We have to define $\Psi(g,r,h) := (g, \Psi h)$ with $\Psi h$ a left $\category{NrmTrivCof}$-map structure for $g$. To do this, let us consider a normal uniform fibration $(f, \phi)$ and a square $(a, b) \co g \rightarrow f$ for which we will construct a lift $\Psi h_f \co B \rightarrow X$ as shown:
\[
\xymatrix{
A \ar[d]_g \ar[r]^a & X \ar[d]^f  \\
B \ar[r]_b \ar@{.>}[ur]|{\Psi h_f} & Y
}
\] 
We first consider the lift $H \co I \times B \rightarrow X$, in the following square (which commutes because the deformation retraction is $0$-oriented), produced by the normal uniform fibration structure of $f$:
\[
\xymatrix{
B \ar[d]_{\delta^0 \times B} \ar[r]^r & A \ar[r]^a & X \ar[d]^f \\
I \times B \ar@{.>}[urr]|H \ar[r]_h & B \ar[r]_b & Y \rlap{.}
}
\]
We define $\Psi h_f \co= H \cdot (\delta^1 \times B)$. That is, the lift $\Psi h_f$ is defined to be $H$ on restricted to the top of the cylinder $I \times B$.  The verification that $f \circ \Psi h_f = b$ is straightforward. 

We now need to check that $\Psi h_f \cdot g = a$, for this we first observe the following diagram:
\[
\xymatrix{
A \ar[r]^g \ar[d]_{\delta^0 \times A} & B \ar[d]|(0.30){\delta^0 \times B} \ar[r]^r & A \ar[r]^a & X \ar[d]^f \\
I \times A \ar[r]_{I \times g} \ar@{.>}[urrr]|(0.25){H_0} & I \times B \ar@{.>}[urr]|H \ar[r]_h & B \ar[r]_b & Y
}
\]
here, the lift $H_0$ is also defined by the uniform fibration structure of $f$, and moreover the triangle created by the lifts commute, since the square on the left is a morphism of left~$\category{UniFib}$-maps. 

We use that $rg = 1_a$ and the strength of the homotopy retraction tuple $(g, r, h)$, to replace the horizontal arrows in the previous diagram in order to obtain the following one:
\[
\xymatrix{
A \ar@{=}[r] \ar[d]_{\delta^0 \times A} & A \ar[rr]^a \ar@{=}[d] && X \ar[d]^f \\
I \times A \ar[r]_{\rho_1} \ar@{.>}[urrr]|(0.25){H_0} & A \ar@{.>}[urr]|a \ar[r]_g & B \ar[r]_b & Y
}
\]
where the lifts cohere by \cref{prop:nufib} using the squash square of~$\bot_A \co \emptyset \rightarrow A$. With this in place,
\begin{flalign*} 
 && \Psi h_f \cdot g &= H \cdot (\delta^1 \times B) \cdot g && \text{(by defn of $\Phi h_f$)} \\
 && 			       &= H \cdot (I \times g) \cdot (\delta^1 \times A) && \text{(by naturality of $\delta^1 \times -$)} \\
 &&			      &= H_0 \cdot (\delta^1 \times A) && \text{(by construction of $H_0$)} \\
 && 			     &= a \cdot \rho_1 \cdot (\delta^1 \times A) && \text{(by normality of $(f, \phi)$)} \\
 &   &			     & =  a	\rlap{,} 
\end{flalign*} 
as required.
\end{proof}

\section{Type-theoretic awfs from normal uniform fibrations} \label{sect:TTNUF}

In order to equip the awfs $(NC_t, NF)$ of normal uniform fibrations with the structure of a type-theoretic awfs we require a functorial Frobenius structure and a stable functorial choice of path objects. In this section, we show how to construct these. 
We focus first with the construction of a stable functorial choice of path objects (sfpo for short), 
cf.~\cref{defn:SFPO}. We work in the context of a suitable topos $(\catE, I, \Em)$ that in addition satisfies hypothesis (M5) from \cref{theorem:main1}. 
Recall from the discussion preceding \cref{theorem:main1} that a suitable topos has a canonical stable and functorial factorisation of the diagonal, called $\lcat{P}_I$, which is constructed via exponentiation by the interval. Our objective is to lift this factorisation to a sfpo. That is, we need to exhibit a lift of $\lcat{P}_I$ as in
\[
\xymatrix{
\category{NrmUniFib} \ar[r]^(0.35){{\lcat{P}_I}} & \category{NrmTrivCof} \times_{\icat{C}} \category{NrmUniFib}  \rlap{.}
}
\]
We can split the problem in two. If we denote by $r, \rho \co \catE^{\rightarrow} \rightarrow \catE^{\rightarrow}$ the two legs of the sfpo (\emph{i.e.}~by composing $\lcat{P}_I$ with the two projections from the pullback), then it is sufficient to show that there are lifts of these functors as in the following diagram.
\[
\xymatrix{
\category{NrmUniFib}  \ar[r]^(0.50){{r}} & \category{NrmTrivCof}   &&
\category{NrmUniFib} \ar[r]^(0.50){{\rho}} & \category{NrmUniFib}  \, .
}
\]
The lift of $r \co \catE^{\rightarrow} \rightarrow \catE^{\rightarrow}$ can easily obtained by collecting some of the results established so far.

\begin{lemma} \label{lem:liftr}
There is a lift of the functor $r \co  \catE^{\rightarrow} \rightarrow \catE^{\rightarrow}$ as shown:
\[
\xymatrix{
\category{NrmUniFib}  \ar[r]^(0.50){{r}} & \category{NrmTrivCof}
}
\]
\end{lemma}

\begin{proof}
We construct the desired lift as the following composite:
\[
\xymatrix{
\category{NrmUniFib} \ar[d] \ar[r] & \category{UniFib} \ar[r] \ar@{=}[d]  & \category{SDR} \ar[d] \ar[r]^(0.40){\Psi} & \category{NrmTrivCof} \ar[d]\\
\catE^{\rightarrow}  \ar@{=}[r] & \catE^{\rightarrow}  \ar[r]_r & \catE^{\rightarrow}  \ar@{=}[r]& \catE^{\rightarrow}  \, .
}
\]
where the lift in the leftmost square is the forgetful functor, that on the middle square comes from \cref{claim:2} and the lift in the rightmost square is the one from \cref{thm:SDRNTC}. 
\end{proof}

\begin{remark} \label{rem:follow-up-on-costructivity}
The proof of \cref{theorem:main1}, which shows that the functor $r \co \catE^{\rightarrow} \rightarrow \catE^{\rightarrow}$ lifts to the category of left maps of the awfs of uniform fibrations, relied crucially on the hypothesis (M6). This says that the image of $r$ lands on the class $\Em$ of generating monomorphisms of the suitable topos. As noted in \cref{note:nonconstruct}, hypothesis (M6) does not hold if we consider $\Em_{dec}$, the class of decidable monos in the context of a presheaf topos. However notice that the proof of \cref{lem:liftr} does not require hypothesis (M6). In other words, the extra `normality' condition on the category of uniform fibrations allows us to get rid of this requirement.
\end{remark}

The construction of the lift for the other functor $\rho \co \catE^{\rightarrow} \rightarrow \catE^{\rightarrow}$ is not quite as direct; we need to briefly recall the construction of the uniform fibration structure produced by \cref{claim:1}. 
Let us consider a map $f \co X \rightarrow Y$ in $\catE$; recall (from the discussion before \cref{theorem:main1}) that the second leg of the factorisation of the diagonal, $\rho_f \co Pf \rightarrow X \times_Y X$ can also be obtained as a pullback of the map $\hhom(i, f)$ where $i \co \partial I \rightarrow I $ stands for the inclusion of the boundary of the interval. 
Let us assume for now that $(f, \theta)$ is a uniform fibration. We know that right orthogonal categories of arrows are closed under pullbacks, thus to give a uniform fibration structure to $\rho_f$ it is sufficient to give one to $\hhom(i, f)$. Now, in order to construct a uniform fibration structure for $\hhom(i,f)$, let us consider a lifting problem with respect to a morphism of the generating category of arrows $\lcat{I}_{\hat{\times}}$ of uniform fibrations; \emph{i.e.}~a square of the form $(U, b) \co \delta^k \htimes j \rightarrow \hhom(i, f)$ where $j \co A \mono B$ is in $\lcat{I}$, for which we  show how to construct a lift. This is shown in the left side of the following diagram:
\[
\xymatrix{
B +_{A} (I \times A) \ar[d]_{\delta^k \htimes j} \ar[r]^(0.65)U & X^{I} \ar[d]^{\hhom(i, f)} &&
\dom(i \htimes (\delta^k \htimes i)) \ar[d]_{i \htimes (\delta^k \htimes j)} \ar[r]^(0.70){{U}} & X \ar[d]^f 
\\
I \times B \ar[r]_b \ar@{.>}[ru]|{\rho\theta_j}  & Y^{I} \times_{Y^{\partial I}} X^{\partial I} &&
I \times (I \times B) \ar[r]_(0.65){{b}} \ar@{.>}[ru]|{\rho\theta_j} & X \, .
}
\]

Transposing along the adjunction $(i \htimes - )  \dashv \hhom(i, -)$ we obtain a square as on the right of the previous diagram. We use that the pushout-product construction is symmetric and associative, and in particular we obtain that $i \htimes (\delta^k \htimes j) \cong \delta^k \htimes (i \htimes j)$. By hypotheis (M5) of the category of generating cofibrations $\lcat{M}$ we know that $i \htimes j$ is a generating monomorphism, thus we find a lift for the square on the right of the previous diagram, denoted by $\rho\theta_i$. By transposing everything back we obtain the desired lift for the original square. This construction produces a uniform fibration structure for $\hhom(i, f)$ which we denote by $\rho\theta$. 


\begin{lemma} \label{lem:liftp}
There is a lift of the functor from \cref{claim:1} as shown:
\[
\xymatrix{
\category{NrmUniFib} \ar[r]^{{\rho}} & \category{NrmUniFib} \, .
}
\] 
\end{lemma}
\begin{proof}
Since the forgetful functor $\category{NrmUniFib} \rightarrow \category{UniFib}$ is fully faithful (\cref{lem:Uff}), and using that right ortogonal categories are closed under pullbacks; it is sufficient to prove that given $(f, \psi)$ a normal uniform fibration, the uniform fibration structure $\rho\psi$ of $\hhom(i, f)$, described in the foregoing discussion, is also normal. 

By \cref{prop:nufib}, we need to show that, for a generating monomorphism $j \co A \mono B$, the lifts in the diagram on the left of the following figure cohere:
\[
\xymatrix{
B +_{A} (I \times A) \ar[d]_{\delta^k \htimes j} \ar[r]^(0.70){\squi_k(j)} & B \ar[r]^{U^*} \ar@{=}[d] & X^{I} \ar[d]^{\hhom(i, f)} &
\dom(\delta^k \htimes (i \htimes j)) \ar[d]_{\delta^k \htimes (i \htimes j)} \ar@{{}{ }{}}@/^0.0pc/[r]^(0.63){\squi_k(i \htimes j)} \ar[r] & I \times B \ar[r]^(0.60){U^*} \ar@{=}[d] & X \ar[d]^f\\
I \times B \ar@{.>}[urr]|(0.30){\rho\theta_{j}} \ar[r]_(0.60){\epsilon \times B} & B \ar[r]_(0.30){b^*} \ar@{.>}[ur]|{U^*} & Y^{I} \times_{Y^{\partial I}} X^{\partial I} &
I \times (I \times B) \ar[r]_(0.60){\epsilon \times (I \times B)} \ar@{.>}[urr]|(0.40){\rho\theta_{j}}& I \times B \ar[r]_(0.65){b^*} \ar@{.>}[ur]|{U^*} & Y \, .
}
\]
by transposing the whole diagram along $ (i \htimes - ) \dashv  \hhom(i, -)$ and using the symmetry and associativity of the pushout-product, we obtain the lifting problem as on the right of the previous diagram, for which we need to show that the lifts cohere. The lift $\rho\theta_{j}$ on the left (on either diagram) is, by construction, the lift obtained from the uniform fibration structure $\rho\theta$ on $\hhom(i, f)$. The result follows by applying \cref{lem:squish}.
\end{proof}


\begin{proposition} \label{thm:PONUF}
Consider a suitable topos $(\catE, I, \Em)$ satisfying condition (M5). Then the stable functorial factorisation of the diagonal $\lcat{P}_I$ lifts to a stable functorial choice of path objects for the awfs of normal uniform fibrations; as shown in the following diagram:
\[
\xymatrix{
\category{NrmUniFib} \ar[d] \ar@{.>}[r]^(0.35){{\lcat{P}_I}} & \category{NrmTrivCof} \times_{\icat{C}} \category{NrmUniFib} \ar[d] \\
\catE^{\rightarrow} \ar[r]_(0.35){\lcat{P}_I} & \catE^{\rightarrow} \times_{\catE} \catE^{\rightarrow}  \, .
}
\]
\end{proposition}

\begin{proof}
This follows by applying \cref{lem:liftr} to lift the functor $r \co \catE^{\rightarrow} \rightarrow \catE^{\rightarrow}$ and by applying \cref{claim:1} and \cref{lem:liftp} to lift the functor $\rho \co \catE^{\rightarrow} \rightarrow \catE^{\rightarrow}$.
\end{proof}

We turn our attention to the proof that the category of arrows of normal uniform fibrations has a functorial Frobenius structure. The structure is given by adapting the functorial Frobenius structure on uniform fibrations constructed in \cite[Theorem~8.8]{Gambino_2017}. Throughout this section, we will work on an arbitrary suitable topos $(\catE, I, \Em)$.

\begin{lemma} \label{lem:squishpb}
Let $i \co A \mono B$ be a monomorphism, and let $f \co X \rightarrow B$ be any map. Then the following holds:
\begin{enumerate}
\item There is an isomorphism 
\[
\delta^k \htimes (f^*i) \cong (I \times f)^*(\delta^k \htimes i)  \, .
\]
\item Pulling back the $k$-squash square of $i$ along the square $(I \times f, f)$ produces the $k$-squash square of $f^*i$; concretely, for $k \in \{0,1\}$, there is an isomorphism:
\[
 \squish_k (f^*i) \cong (I \times f, f)^*(\squish_k (i))  \, .
\]
\end{enumerate}
\end{lemma}

\begin{proof}
To show item $(1)$, let us first consider the following cube:
\[
\xymatrix{
f^*A \ar[dd]|{\delta^k \times f^*A} \ar[dr]^{f^*i} \ar[rr]^{\pi} && A \ar[dr]^i \ar@{.>}[dd]|(0.49)\hole|(0.30){\delta^k \times A} & \\
& X \ar[rr]^(0.35)f \ar[dd]|(0.35){\delta^k \times X} && B \ar[dd]|{\delta^k \times B} \\
I \times (f^*A) \ar@{.>}[rr]|(0.53)\hole_(0.75){I \times \pi} \ar[dr]_{I \times f^*i} && I \times A \ar@{.>}[dr]^{I \times i} & \\
& I \times X \ar[rr]_{I \times f} && I \times B \rlap{.}
}
\]
Here, the square on the top is the pullback of $i$ along $f$. It is straightforward to verify that all squares pointing from left to right are Cartesian, and note that the squares on the left and right are the outer squares used for defining the pushout-products $\delta^k \htimes (f^*i)$ and $\delta^k \htimes i$ respectively. All of this implies that there is a comparison map $\delta^k \htimes (f^*i) \rightarrow (I \times f)^*(\delta^k \htimes i)$, which is an isomorphism because colimits in $\catE$ are universal. Item $(2)$ follows directly form item $(1)$.
\end{proof}

We recall a result about the squares $\theta^k \htimes i \co i \rightarrow \delta^k \htimes i$ from \cite[Lemma~4.3]{Gambino_2017}. 

\begin{lemma} \label{lem:thetacart}
For every $i \co A \rightarrow B$, the square $\theta^k \htimes i \co i \rightarrow \delta^k \htimes i$ below is Cartesian. 
\[
\xymatrix{
 A \ar[d]_i \ar[r] & B +_A (I \times A) \ar[d]^{\delta^k \htimes i} \\
 B \ar[r]_{\delta^{1-k} \times B} & I \times B  \, .
}
\]
\end{lemma}

\begin{proof}
The proof uses once again the fact that colimits in $\catE$ are universal. Let us compute the pullback of $\delta^k \htimes i$ along $\delta^{1-k} \times B$. By universality of colimits, this is the same as pulling back the diagram defining $B +_A (I \times A)$ and then calculating the colimit. 

We can observe in the following picture, the result of first pulling back the defining diagram of $B +_A (I \times A)$ which appears as the upper span of the right-most square on the following cube:
\[
\xymatrix{
\emptyset \ar[dd]|{} \ar[dr]^{} \ar[rr]^{} && A \ar[dr]^{i} \ar@{.>}[dd]|(0.49)\hole|(0.30){\delta^k \times A} & \\
& \ar[dd] \emptyset \ar[rr] && B \ar[dd]|{\delta^k \times B} \\
A \ar@{.>}[rr]_(0.65){\delta^{1-k} \times A} \ar[dr]_{i} && I \times A \ar@{.>}[dr]^{I \times i} & \\
& B \ar[rr]_{\delta^{1-k} \times B}  && I \times B  \, .
}
\]
Let us notice that the pullback of $\delta^k \times B$ (respectively $\delta^k \times A$) along $\delta^{1-k} \times B$ (respectively $\delta^{1-k} \times A$) is empty since the interval has disjoint endpoints. We conclude that the colimit of the upper span of the left-most square on the cube must be equal to $A$ and moreover, the universal arrow down to $B$ has to be $i \co A \rightarrow B$.
\end{proof}

Consider a generating monomorphism $i \co A \mono B$ and a uniform fibration $f \co X \rightarrow B$, then there are two possible trivial uniform cofibration structures on the map $\delta^k \htimes (f^*i)$: the first one is the canonical one, \emph{i.e.}~the one given by the fact that $f^*i$ is also a generating monomorphism. The second one is the one provided by the functorial Frobenius structure on uniform fibrations using the isomorphism $\delta^k \htimes (f^*i) \cong (I \times f)^*(\delta^k \htimes i)$ of \cref{lem:squishpb}. These two are actually the same structure as we show in the following lemma.

\begin{lemma} \label{lem:strpb}
Consider $i \co A \mono B$ be a generating monomorphism and $f \co X \rightarrow B$ a uniform fibration.  Then the two possible trivial uniform cofibration structures on $\delta^k \htimes (f^*i)$ coincide. 
\end{lemma}

\begin{proof}
Let us denote by $\lambda^1$ and $\lambda^2$, respectively, the canonical trivial uniform cofibration structure on $\delta^k \htimes (f^*i)$ and the one obtained by applying the functorial Frobenius structure.
In order to prove they are the same, let us consider $g \co Z \rightarrow Y$ a uniform fibration and a square $(a, b) \co \delta^k \htimes (f^*i) \rightarrow g$. Without loss of generality, let us denote by $\lambda^1$ and $\lambda^2$ the two fillers of this square given by the uniform trivial cofibration structure with the same name. 
We have to show that $\lambda^1 = \lambda^2$. If we go over the proof of \cite[Proposition~8.8]{Gambino_2017}, just before the conclusion, a retract diagram is
used to transfer the structure of a trivial cofibration to the desired morphism (since trivial cofibrations are closed under retracts). In our situation, this retract diagram is given by the two left-most squares shown below:
\[
\xymatrix{
\cdot \ar[d]_{\delta^k \htimes (f^*i)} \ar[rr] && \cdot \ar[d]|{\delta^k \htimes \delta^k \htimes (f^*i)} \ar[rr] && \cdot \ar[d]|{\delta^k \htimes (f^*i)} \ar[rr]^a && Z \ar[d]^g\\
\cdot \ar[rr]_{t} && \cdot \ar[rr] && \cdot \ar[rr]_b && Y \, , 
}
\]
where the left-most square is $\theta^k \htimes \delta^k \htimes (f^*i)$. Notice that $\delta^k \htimes \delta^k \htimes (f^*i)$ has a canonical trivial cofibration structure and thus, the square $\delta^k \htimes \delta^k \htimes (f^*i) \rightarrow f$ has a lift which we denote by $\lambda$. By definition, the lift $\lambda^2$ is equal to $\lambda \cdot t$ where $t$ is the horizontal arrow on the lower left part of the diagram. 

On the other hand, the lift of the outer square of the previous diagram is $\lambda^1$. Thus if we want to show that $\lambda^1 = \lambda^2$ it is sufficient to show that the square $\theta^k \htimes \delta^k \htimes(f^*i)$ is a morphism of trivial uniform cofibrations. To show this, we use that the pushout-product is symmetric and associative, and thus $\theta^k \htimes \delta^k \htimes(f^*i) \cong \delta^k \htimes \theta^k \htimes (f^*i)$. From this, we see that the square is a morphism of trivial uniform cofibrations provided the square $\theta^k \htimes (f^*i)$ is a morphism of generating monomorphisms, \emph{i.e.}~if it is Cartesian. But this is precisely the statement of \cref{lem:thetacart}. 
\end{proof}

We wish to show that the functorial Frobenius structure on uniform fibrations of~\cite[Theorem~8.8]{Gambino_2017} can be extended to normal uniform fibrations. We start with a proposition.

\begin{proposition} \label{lem:main1}
There is a lift of the pullback functor as shown:
\[
\xymatrix{
\nufib \times_{\catE} \category{UniFib} \ar[d] \ar[r]^(0.50){{PB}} & \category{NrmTrivCof} \ar[d] \\
\catE^{\rightarrow} \times_{\catE} \catE^{\rightarrow} \ar[r]_(0.60){PB} & \catE^{\rightarrow} \, .
}
\]
\end{proposition}

\begin{proof}
Object-wise, this follows directly from \cite[Theorem~8.8]{Gambino_2017}. To see this, we notice that there are no more objects in the image of the category of arrows $\nufib$ than in the image of $\ufib$ thus we can apply the functorial Frobenius structure for uniform fibrations. Then we use the functor $\category{TrivCof} \rightarrow \category{NrmTrivCof}$, obtain by functoriality of the left orthogonal functor $\lorth{(-)}$ applied to the forgetful functor $\category{NrmUniFib} \rightarrow \category{UniFib}$.

For the morphism case, we first notice that the only morphisms in $\nufib$ that we need to consider are the squash squares. Thus let us consider a cospan of squares as in the following diagram:
\[
\xymatrix@R=0.6cm{
&& \cdot \ar[rd]^{\squi_k(i)} \ar[dd]_{\delta^k \htimes i} & \\
&&& B \ar@{=}[dd] \\
X' \ar[rr]^{f'} \ar[dr]_{m} && I \times B \ar[dr]_{\epsilon \times B} & \\
& X \ar[rr]_{f} && B  \, .
}
\]
such that the vertical square is the squash square of a generating monomorphism $i \co A \mono B$ and the horizontal square is a morphism of uniform fibrations $(m , \epsilon \times B) \co f' \rightarrow f$. We need to verify that pulling back the squash square along the morphism of uniform fibrations is a morphism of normal trivial cofibrations. 

The first thing we do is to split this cospan of squares into two, by factoring through the pullback square of $f$ along $\epsilon \times B$. That is we obtain the following diagrams:
\[
\xymatrix@R=0.6cm{
&& \cdot \ar[rd]^{\squi_k(i)} \ar[dd]_{\delta^k \htimes i} & &&
	&& \cdot \ar@{=}[rd] \ar[dd]_{\delta^k \htimes i} &\\
&&& B \ar@{=}[dd] && 
	&&& \cdot \ar[dd]^{\delta^k \htimes i} && \\
I \times X \ar[rr]^{I \times f} \pbcorner \ar[dr]_{\epsilon \times X} && I \times B \ar[dr]_{\epsilon \times B} & &&
	X' \ar[rr]^{f'} \ar@{.>}[dr]_{m^*} && I \times B \ar@{=}[dr] & \\
& X \ar[rr]_{f} && B && 
	& I \times X \ar[rr]_{I \times f} && I \times B \, , 
}
\]
where the dotted arrow $m^* \co X' \rightarrow (I \times X)$ is obtain by universal property. Notice that composing the two cospans of squares along their common face, produces the original one. Notice also that the two horizontal squares are morphisms of uniform fibrations. 

Let us focus first on the cospan of the right. The identity morphism $1 \co (\delta^k \htimes i) \rightarrow (\delta^k \htimes i)$ is a morphism of trivial uniform cofibrations, thus if we pull-back this along the morphism of uniform fibrations $(f', I \times f) \co m^* \rightarrow 1_{\delta^k \htimes i}$ we obtain a morphism of trivial uniform cofibrations by \cite[Theorem~8.8]{Gambino_2017} to which we can apply the functor $\category{TrivCof} \rightarrow \category{NrmTrivCof}$ to obtain a morphism of normal trivial cofibrations.  

With this we have reduced the situation to the cospan of squares on the left of the previous diagram. Using item $(2)$ of \cref{lem:squishpb} we see that the pullback of the squash square of $i \co A \mono B$ along the square $(I \times f, f) \co \epsilon \times X \rightarrow \epsilon \times B$ is the squash square of $f^*i \co f^*A \mono X$. This square is a morphism in $\nufib$ provided that the canonical trivial normal cofibration structure of $\delta^k \htimes (f^*i)$ is the same as that obtained from the functorial Frobenius structure; but this follows from \cref{lem:strpb}.
\end{proof}


\begin{proposition} \label{thm:FFNUF}
Let $(\catE, I, \Em)$ be a suitable topos. Then the awfs $(NC_t, NF)$ of normal uniform fibrations has a functorial Frobenius structure.
\end{proposition}

\begin{proof}
Using the lift of \cref{lem:main1} and the forgetful functor $\category{NrmUniFib} \rightarrow \category{UniFib}$, we find a lift of the pullback functor as one shown: 
\[
\xymatrix{
\nufib \times_{\catE} \category{NrmUniFib} \ar[d] \ar[r]^(0.58){{PB}} & \category{NrmTrivCof} \ar[d] \\
\catE^{\rightarrow} \times_{\catE} \catE^{\rightarrow} \ar[r]_(0.60){PB} & \catE^{\rightarrow}  \, .
}
\]
The fact that we can extend this structure from $\nufib$ to the whole category $\category{NrmTrivCofi}$ follows from \cite[Proposition~6.8]{Gambino_2017}.
\end{proof}


\begin{theorem} \label{thm:mainNUF}
Consider a suitable topos $(\catE, I, \Em)$ satisfying condition (M5). Then the awfs $(NC_t, NF)$ of normal uniform fibrations has the structure of a type-theoretic awfs.
\end{theorem}

\begin{proof}
The claim follows from \cref{thm:PONUF} and \cref{thm:FFNUF}.
\end{proof}

\appendix

\section{Some technical definitions}
\label{sect:appendix}

Let $(\catC, \rho, \chi)$ be a comprehension category.  For $n \in \mathbb{N}$, the category $DT_n(\rho, \chi)$ of \emph{dependent $n$-tuples} over $(\catC, \rho, \chi)$ is defined as follows. 
Objects are tuples $(\Gamma, A_1, \dots, A_n)$ where $\Gamma$ is an element in the base category, $A_1$ is in the fibre of $\rho$ over $\Gamma$ and,  for $i > 1$, $A_i$ is in the fiber over $\Gamma . A_1 . \cdots . A_{i-1}$.  An arrow $(\Delta, B_1, \dots , B_n) \rightarrow (\Gamma, A_1, \dots, A_n)$ consists of a tuple of the form $(u, f_1, \dots, f_n)$ where $f_1 : B_1 \rightarrow A_1$ is over $u : \Delta \rightarrow \Gamma$ and for $i>1$ we have
\[
\xymatrix{
B_i \ar[rr]^{f_i} \ar@{.}[d] && A_i \ar@{.}[d] \\
\Delta.B_1.\dots.B_{i-1} \ar[rr]^{f_{i-1}} && \Gamma.A_1.\dots. A_{i-1} \rlap{.}
}
\]
Composition and identities are given component-wise by the structure of the fibration~$\rho$.  We say that an arrow of dependent tuples is \emph{Cartesian} if every composing arrow (except the one of the base category) is Cartesian with respect to $\rho$. 

\subsection*{$\Sigma$-types and $\Pi$-types} A {\em choice of $\Sigma$-types} for $(\catC, \rho, \chi)$ consists of an operation that assigns to each dependent tuple $(\Gamma, A, B) \in DT_2(\rho, \chi)$ a tuple $(\Sigma_A B, pair_{A, B}, sp_{A, B})$ consisting of the following data:
\begin{enumerate}
\item $\Sigma_A B$ is an object of $\catE$ over $\Gamma$.

\item $pair_{A,B}$ is an arrow over $\chi_A$ as shown:
\[
\xymatrix{
\Gamma.A.B \ar[rr]^{pair_{A, B}} \ar[d]_{\chi_{B}} && \Gamma.\Sigma_A B \ar[d]^{\chi_{\Sigma_A B}} \\
\Gamma.A \ar[rr]_{\chi_A} && \Gamma \rlap{.}
}
\]

\item $sp_{A,B}$ is an operation that takes a dependent tuple $(\Gamma, \Sigma_A B, C) \in DT_2(\chi, \rho)$ and a section $t$ of $C$ over $pair_{A, B}$, as in the following solid arrowed diagram:
\[
\xymatrix{
&& \Gamma.\Sigma_A B.C \ar[d] \\
\Gamma.A.B \ar[rr]_{pair_{A,B}} \ar[rru]^{t} && \Gamma.\Sigma_A B \ar@/_2.0pc/@{.>}[u]_{sp_{A,B}(C,t)}
}
\]
to a section $sp_{A,B}(C, t)$ of $C$, shown in the above diagram as the dotted arrow.

\item The above data is subject to the condition that, for any section~$t$ of~$C$ over~$pair_{A, B}$,
$sp_{A,B}(C,t) \circ pair_{A,B} = t$. Thus says that the triangle in the diagram of item $(3)$ involving the dotted arrow commutes.
\end{enumerate}


A {\em choice of $\Pi$-types} for $(\catC, \rho, \chi)$ consists of an operation that assigns to each dependent tuple $(\Gamma, A, B) \in DT_2(\rho, \chi)$ a tuple $(\Pi_A B, \lambda_{A, B}, app_{A, B})$ consisting of the following data:
\begin{enumerate}
\item $\Pi_A B$ is an object of  $\catE$  over $\Gamma$.
\item $\lambda_{A, B}$ is an operation that takes a section $t : \Gamma.A \rightarrow \Gamma.A.B$ of $\chi_B$ to a section $\lambda_{A, B}(t) : \Gamma \rightarrow \Pi_A B$ of $\chi_{\Pi_A B}$, as shown in the following diagram:
\[
\xymatrix{
 & \Gamma.A.B \ar[d] && & \Gamma.\Pi_A B \ar[d] \\
\Gamma.A \ar@{=}[r] \ar[ur]^t & \Gamma.A \ar@{}[rru]|{\mapsto}&& \Gamma \ar@{=}[r] \ar[ur]^{\lambda_{A, B}(t)} & \Gamma \rlap{.}
}
\]
\item $app_{A, B}$ is an arrow in the slice over $\Gamma.A$, as shown:
\[
\xymatrix{
\Gamma.A.\Pi_A B \ar[rr]^{app_{A, B}} \ar[d]_{\chi_{\Pi_A B}} && \Gamma.A.B \ar[d]^{\chi_B} \\
\Gamma.A \ar@{=}[rr] && \Gamma.A  \rlap{,}
}
\]
where $\Gamma.A.\Pi_A B$ is (the comprehension of) \emph{any} reindexing of $\Pi_A B$ along $\chi_A$. Notice that the choice of $app_{A, B}$ determines uniquely any other choice with respect to a different Cartesian reindexing of $\Pi_A B$, this follow by the universal property of Cartesian arrows.

\item This data must be subject to the condition that, for any section $t \co \Gamma.A \rightarrow \Gamma.A.B$ of $\chi_B$, $app_{A, B} \circ (\lambda(t)[\chi_A]) = t$. 
where $\lambda(t)[\chi_A]$ is the result of reindexing $\lambda(t)$ along $\chi_A$.
\end{enumerate}


\subsection*{Strict stability}  Let $(\catC, \rho, \chi)$ be a split comprehension category.

A choice of $\Sigma$-types $(\Sigma, pair, sp)$ for $(\catC, \rho, \chi)$ is said to be {\em strictly stable} if for every morphism $\sigma: \Delta \rightarrow \Gamma$ in the base category and for any dependent tuple $(\Gamma, A, B) \in DT_2(\rho, \chi)$, the following conditions are satisfied:
\begin{enumerate}
\item $\Sigma_{A[\sigma]} B[\sigma] = (\Sigma_A B) [\sigma]$

\item The following diagram commutes:
\[
\xymatrix{
\Delta.A[\sigma].B[\sigma] \ar[rr]^{\sigma^{**}} \ar[d]_{pair_{A[\sigma], B[\sigma]}} && \Gamma.A. B \ar[d]^{pair_{A, B}} \\
\Delta.\Sigma_{A[\sigma]} B[\sigma] \ar[rr]_{\sigma^{*}} && \Gamma.\Sigma_{A} B \rlap{,}
}
\]
where the horizontal arrows are obtained by the split reindexing along $\sigma$. 

\item For any dependent tuple $(\Gamma, \Sigma_{A} B, C)$ in $DT_2(\rho, \chi)$, and any section $t$ of $C$ over $pair_{A, B}$ there is a corresponding dependent tuple $(\Delta, \Sigma_{A[\sigma]} B[\sigma], C[\sigma])$ and a section $t[\sigma]$ of $C$ over $pair_{A[\sigma], B[\sigma]}$ obtained by reindexing. The following diagram is required to commute:
\[
\xymatrix{
\Delta.\Sigma_{A[\sigma]}B[\sigma] \ar[rr]^{\sigma^{*}} \ar[d]_{sp(C[\sigma], t[\sigma])} && \Gamma.\Sigma_AB \ar[d]^{sp(C, t)} \\
\Delta.\Sigma_{A[\sigma]} B[\sigma].C[\sigma] \ar[rr]_{\sigma^{**}} && \Gamma.\Sigma_A B.C \rlap{,}
}
\]
where the horizontal arrows are obtained by the split reindexing along $\sigma$. 
\end{enumerate}

 A choice of $\Pi$-types $(\Pi, \lambda, app)$ for $(\catC, \rho, \chi)$ is said to be {\em strictly stable} if for every morphism $\sigma: \Delta \rightarrow \Gamma$ in the base category and for any dependent tuple $(\Gamma, A, B) \in DT_2(\rho, \chi)$, the following conditions are satisfied:
\begin{enumerate}
\item $\Pi_{A[\sigma]} B[\sigma] = (\Pi_A B) [\sigma]$
\item For any section $t$ of $\chi_B$ there is a corresponding section $t[\sigma]$ of $\chi_{B[\sigma]}$ obtained by reindexing. This two sections must be related by the following commutative diagram:
\[
\xymatrix{
\Delta \ar[rr]^{\sigma} \ar[d]_{\lambda_{A[\sigma], B[\sigma]}(t[\sigma])} && \Gamma \ar[d]^{\lambda_{A, B}(t)} \\
\Delta.\Pi_{A[\sigma]} B[\sigma] \ar[rr]_{\sigma^{*}} && \Gamma.\Pi_A B\rlap{,}
}
\]
where the lower horizontal arrow is obtained by the split reindexing along $\sigma$. 
\item The following diagram commutes:
\[
\xymatrix{
\Delta.A[\sigma].\Pi_{A[\sigma]} B[\sigma] \ar[rr]^{\sigma^{**}} \ar[d]_{app_{A[\sigma], B[\sigma]}} && \Gamma.A.\Pi_A B \ar[d]^{app_{A, B}} \\
\Delta.A[\sigma].B[\sigma] \ar[rr]_{\sigma^{**}} && \Gamma.A.B\rlap{,}
}
\]
where the horizontal arrows are obtained by split reindexing along $\sigma$. 
\end{enumerate}

\subsection*{Pseudostability} Let $(\catC, \rho, \chi)$ be a comprehension category.

A choice of $\Sigma$-types $(\Sigma, pair, sp)$ for $(\catC, \rho, \chi)$  is said to be {\em pseudo-stable} if for every Cartesian arrow $(\sigma, f, g) : (\Delta, A', B') \rightarrow (\Gamma, A, B)$ of dependent tuples, the following conditions are satisfied:
\begin{enumerate}
\item There is a Cartesian arrow $\Sigma_f g : \Sigma_{A'} B' \rightarrow \Sigma_A B$ over $\sigma$ and the assignment:
\[
(\sigma, f, g) \mapsto (\sigma, \Sigma_f g)
\]
is functorial, i.e. $\Sigma_{1_A} 1_B = 1_{\Sigma_A B}$ and $\Sigma_{(f' \circ f)} (g' \circ g) = \Sigma_{f'} g' \circ \Sigma_f g$. 

\item The following diagram commutes:
\[
\xymatrix{
\Delta.A'.B' \ar[rr]^{g} \ar[d]_{pair_{A', B'}} && \Gamma.A. B \ar[d]^{pair_{A, B}} \\
\Delta.\Sigma_{A'} B' \ar[rr]_{\Sigma_f g} && \Gamma.\Sigma_{A} B \rlap{.}
}
\]

\item For any Cartesian arrow $h : C' \rightarrow C$ above $\Sigma_f g : \Sigma_{A'} B' \rightarrow \Sigma_A B$ and for any section $t$ of $C$ over $pair_{A, B}$ there is a corresponding section $t'$ of $C'$ over $pair_{A', B'}$ obtained by reindexing. The following diagram is required to commute:
\[
\xymatrix{
\Delta.\Sigma_{A'}.B' \ar[rr]^{\Sigma_fg} \ar[d]_{sp(C',t')} && \Gamma.\Sigma_AB \ar[d]^{sp(C, t)} \\
\Delta.\Sigma_{A'} B'.C' \ar[rr]_{h} && \Gamma.\Sigma_A B.C \rlap{.}
}
\]
\end{enumerate}

A choice of $\Pi$-types $(\Pi, \lambda, app)$ for $(\catC, \rho, \chi)$ is said to be {\em pseudo-stable} if for every Cartesian arrow $(\sigma, f, g) : (\Delta, A', B') \rightarrow (\Gamma, A, B)$ of dependent tuples, the following conditions are satisfied:
\begin{enumerate}
\item There is a Cartesian arrow $\Pi_f g : \Pi_{A'} B' \rightarrow \Pi_A B$ over $\sigma$ and the assignment:
\[
(\sigma, f, g) \mapsto (\sigma, \Pi_f g)
\]
is functorial, i.e. $\Pi_{1_A} 1_B = 1_{\Pi_A B}$ and $\Pi_{(f' \circ f)} (g' \circ g) = \Pi_{f'} g' \circ \Pi_f g$. 
\item For any section $t : \Gamma.A \rightarrow \Gamma.A.B$ of $B$ there is a corresponding section $t' : \Delta.A' \rightarrow \Delta.A'.B'$ of $B'$ obtained by reindexing along $f : \Delta.A' \rightarrow \Gamma.A$. Then, the following diagram commutes:
\[
\xymatrix{
\Delta \ar[rr]^{\sigma} \ar[d]_{\lambda_{A', B'}(t')} && \Gamma \ar[d]^{\lambda_{A, B}(t)} \\
\Delta.\Pi_{A'} B' \ar[rr]_{\Pi_f g} && \Gamma.\Pi_A B \rlap{.}
}
\]
\item The following diagram commutes:
\[
\xymatrix{
\Delta.A'.\Pi_{A'} B' \ar[rr]^{\Pi_f g} \ar[d]_{app_{A', B'}} && \Gamma.A.\Pi_A B \ar[d]^{app_{A, B}} \\
\Delta.A'.B' \ar[rr]_{g} && \Gamma.A.B \rlap{.}
}
\]
\end{enumerate}

\bibliographystyle{plain}

\newcommand{\noopsort}[1]{}

\end{document}